\documentclass{article}

\usepackage{a4wide}

\usepackage{amsmath,amssymb,amsthm}

\usepackage{amssymb}

\usepackage{mathtools}

\usepackage[english]{babel}

\usepackage{hyperref}

\usepackage[title]{appendix}

\usepackage{cleveref}

\usepackage{subfig}

\usepackage{bbm}

\numberwithin{equation}{section}

\usepackage{siunitx}

\usepackage{booktabs}

\usepackage{enumerate}

\usepackage[dvipsnames]{xcolor}
\definecolor{myred}{HTML}{FF3D3D}
\definecolor{mycyan}{HTML}{0474BE}
\definecolor{mygreen}{HTML}{1EB6D6}

\usepackage{tikz}
\usepackage{pgfplots}
\pgfplotsset{width=.36\textwidth,compat=1.12}

\usepgfplotslibrary{fillbetween}

\newcommand{\hili}[2][mycyan]{\textcolor{#1}{#2}}

\usepackage{stmaryrd}
\newcommand{\jump}[1]{\llbracket #1 \rrbracket}

\usepackage{MnSymbol}
\newcommand{\jumprecon}[1]{\llangle #1 \rrangle}

\usepackage{units}

\newcommand{\R}{\mathbb{R}}
\newcommand{\N}{\mathbb{N}}
\newcommand{\Z}{\mathbb{Z}}
\newcommand{\eps}{\varepsilon}
\newcommand{\dist}{\operatorname{dist}}
\newcommand{\OpR}{\mathcal{R}}
\newcommand{\cell}{\mathcal{C}}
\newcommand{\timecell}{\mathcal{T}}
\newcommand{\I}{\mathcal{I}}
\newcommand{\minmod}{\operatorname{minmod}}
\newcommand{\OpLL}{\mathcal{L}_{\rm L}}
\newcommand{\OpLNL}{\mathcal{L}_{\rm NL}}
\newcommand{\sign}{\operatorname{sign}}
\newcommand{\supp}{\operatorname{supp}}
\newcommand{\Woneinfty}{\rm{W}^{1,\infty}(\R)}
\newcommand{\Linfty}{\rm{L}^{\infty}}
\newcommand{\Lone}{\rm{L}^1}
\newcommand{\Loneloc}{\Lone_{\operatorname{loc}}}
\newcommand{\TV}{\operatorname{TV}}
\newcommand{\BV}{\rm{BV}}

\newcommand{\Dx}{{\Delta x}}
\newcommand{\Dt}{{\Delta t}}
\newcommand{\hf}{{\unitfrac{1}{2}}}
\newcommand{\thf}{{\unitfrac{3}{2}}}
\newcommand{\jphf}{{j+\hf}}
\newcommand{\jmhf}{{j-\hf}}
\newcommand{\jpkmhf}{{j+k-\hf}}
\newcommand{\jmkphf}{{j-k+\hf}}
\newcommand{\jmlphf}{{j-l+\hf}}
\newcommand{\jplmhf}{{j+l-\hf}}
\newcommand{\jplmthf}{{j+l-\thf}}
\newcommand{\jmlpthf}{{j-l+\thf}}
\newcommand{\jpkphf}{{j+k+\hf}}

\newcommand{\Approx}{\operatorname{L}}

\newcommand*\diff{\mathop{}\!\mathrm{d}}
\newcommand{\ind}{{\mathbbm{1}}}		

\newtheorem{theorem}{Theorem}[section]
\newtheorem{lemma}[theorem]{Lemma}
\newtheorem{proposition}[theorem]{Proposition}
\newtheorem{definition}[theorem]{Definition}
\newtheorem{rem}[theorem]{Remark}

\title{Second-order accurate TVD numerical methods for\\ nonlocal nonlinear conservation laws}

\author{\textsc{Ulrik S.~Fjordholm}\thanks{Department of Mathematics, University of Oslo, Norway (\url{ulriksf@math.uio.no})}
\and \textsc{Adrian\,M.~Ruf}\thanks{Seminar for Applied Mathematics, Department of Mathematics, ETH Z\"urich, Switzerland (\url{adrian.ruf@sam.math.ethz.ch})}}
\date{}

\begin{document}

\maketitle
\begin{abstract}
We present a second-order accurate numerical method for a class of nonlocal nonlinear conservation laws called the "nonlocal pair-interaction model" which was recently introduced by Du, Huang, and LeFloch. Our numerical method uses second-order accurate reconstruction-based schemes for local conservation laws in conjunction with appropriate numerical integration. We show that the resulting method is total variation diminishing (TVD) and converges towards a weak solution. In fact, in contrast to local conservation laws, our second-order reconstruction-based method converges towards the unique entropy solution provided that the nonlocal interaction kernel satisfies a certain growth condition near zero. Furthermore, as the nonlocal horizon parameter in our method approaches zero we recover a well-known second-order method for local conservation laws. In addition, we answer several questions from the paper from Du, Huang, and LeFloch concerning regularity of solutions. In particular, we prove that any discontinuity present in a weak solution must be stationary and that, if the interaction kernel satisfies a certain growth condition, then weak solutions are unique. We present a series of numerical experiments in which we investigate the accuracy of our second-order scheme, demonstrate shock formation in the nonlocal pair-interaction model, and examine how the regularity of the solution depends on the choice of flux function.
%
%
\end{abstract}





\section{Introduction}
\subsection{Objective of the paper}
We consider the `nonlocal pair-interaction model'
\begin{gather}
\begin{aligned}
	\frac{\partial u}{\partial t} + \int_0^{\delta}\frac{g(u,\tau_h u)-g(\tau_{-h}u,u)}{h}\omega_{\delta}(h)\diff h = 0,& & &(x,t)\in\R\times(0,T),\\
	u(x,0) = u_0(x),& &&x\in\R,
\end{aligned}
\label{nonlocal model}
\end{gather}
which is a nonlocal variant of the (local) scalar conservation law
\begin{gather}
\begin{aligned}
	\frac{\partial u}{\partial t} +\frac{\partial f(u)}{\partial x} = 0,& & &(x,t)\in\R\times(0,T),\\
	u(x,0) = u_0(x),& &&x\in\R.
\end{aligned}
\label{local claw}
\end{gather}
Here, $g$ is a two-point, monotone flux function that is consistent with the local flux $f$ in the sense that $g(u,u)=f(u)$, $\omega_\delta$ is a kernel characterizing nonlocal interactions, and $\tau_{\pm h} u(x,t) = u(x\pm h,t)$ denotes the standard shift operator in space. The nonlocal pair-interaction model was introduced by Du, Huang, and LeFloch in \cite{du2017nonlocal} where the authors established existence and uniqueness of entropy solutions. Their global existence result is based on the convergence of a first-order accurate finite volume method inspired by first-order finite volume methods for (local) conservation laws. In \cite{du2017numerical} Du and Huang further presented numerical experiments for this first-order scheme.


Our first goal is to \emph{design a second-order accurate numerical method for the nonlocal model}~\eqref{nonlocal model} that is asymptotically compatible with a second-order scheme for the local conservation law~\eqref{local claw}. The method we construct is based on second-order accurate reconstruction-based schemes coupled with a trapezoidal rule to numerically approximate the weighted integral in~\eqref{nonlocal model}.

Our second goal is to show that \emph{solutions of the nonlocal model have more regularity} as compared to solutions of local conservation laws. Specifically, we will show that \emph{weak solutions of~\eqref{nonlocal model} can only exhibit stationary discontinuities}. This improved regularity of the nonlocal model substantiates the usefulness and practicality of higher-order schemes like the one presented here.

\subsection{Background on the nonlocal pair-interaction model}

The nonlocal pair-interaction model is a very recent contribution to the problem of modeling nonlocal convection (see \cite{du2012new} for an extensive overview of other contributions). One key feature of the nonlocal pair-interaction model which many other models do not share is the explicit use of the nonlocal horizon parameter $\delta$ to characterize nonlocal interactions. This is inspired by the same notion used in peridynamics, see \cite{silling2000reformulation}. Another feature of the nonlocal pair-interaction model is that, as the nonlocal horizon parameter vanishes, the nonlocal model \eqref{nonlocal model} reduces to the local conservation law \eqref{local claw} \cite{du2017numerical}. This is to be contrasted to other nonlocal models which do not enjoy this property, see e.g. \cite{Colombo2019}.

A different nonlocal, nonlinear model with interactions over a finite horizon was proposed in \cite{du2012new}; however, only local existence results could be established due to the lack of a maximum principle.
The nonlocal pair-interaction model, on the other hand, enjoys the maximum principle and generally shares many properties of local conservation laws, see \cite{du2017nonlocal}.

Let $\Dx>0$ be the spatial discretization parameter and let $x_j=j\Dx$ and $x_\jphf = (\jphf)\Dx$ denote the midpoints and endpoints of the spatial grid cells $\cell_j=(x_\jmhf,x_\jphf)$. In order to show existence of solutions to the nonlocal pair-interaction model Du, Huang, and LeFloch \cite{du2017nonlocal} used the numerical scheme
\begin{gather}
\begin{aligned}
	u_j^{n+1} &= u_j^n - \Dt \smashoperator[l]{\sum_{k=1}^{\max\{r,1\}}}\frac{g(u_j^n,u_{j+k}^n) - g(u_{j-k}^n,u_j^n)}{k\Dx} W_k,\\
	u_j^0 &= \frac{1}{\Dx}\int_{\cell_j} u_0(x)\diff x,
\end{aligned}
\label{first-order numerical scheme}
\end{gather}
where $r=\lfloor\frac{\delta}{\Dx}\rfloor$ and the weights $W_k$ are given by
\begin{equation*}
	W_k = \int_{(k-1)\Dx}^{k\Dx} \omega_\delta(h)\diff h + \ind_{k=r}\int_{r\Dx}^\delta \omega_\delta(h)\diff h,\qquad k=1,\ldots,r.
\end{equation*}
By keeping the spatial grid size $\Dx$ fixed and letting $\delta\to 0$, the first equation in \eqref{first-order numerical scheme} reduces to the standard monotone finite volume scheme
\begin{equation}
	u_j^{n+1} = u_j^n - \frac{\Dt}{\Dx}\left(g(u_j^n,u_{j+1}^n) - g(u_{j-1}^n,u_j^n)\right)
	\label{first-order local numerical scheme}
\end{equation}
for the local conservation law \eqref{local claw}.
The nonlocal scheme \eqref{first-order numerical scheme} admits an analysis very similar to that of standard monotone schemes, detailed for example in \cite{crandall1980monotone}.

Moreover, as both $\delta$ and $\Dx$ vanish the scheme \eqref{first-order numerical scheme} converges to the entropy solution of the local conservation law \eqref{local claw}, see \cite{du2017nonlocal}. This leads to the so-called asymptotic compatibility of the numerical scheme, as defined in \cite{tian2014asymptotically}, for the nonlocal model \eqref{nonlocal model}.

\subsection{Background on second-order TVD schemes for local conservation laws}

The numerical scheme \eqref{first-order numerical scheme}, studied in \cite{du2017nonlocal,du2017numerical}, shares the drawback of the monotone finite volume method \eqref{first-order local numerical scheme} for local conservation laws of being at most first-order accurate, see e.g.\ \cite{harten1976finite}. In the case of conservation laws, one popular way of increasing the order of accuracy is to use higher-order reconstructed approximations instead of piecewise constant values in monotone schemes. This stems from an idea by van Leer, see \cite{van1979towards}.

Given cell averages $u_j$ at time $t$ defining a spatially piecewise constant function $u_\Dx(x,\cdot) = u_j$, one can construct a piecewise linear function
\begin{equation*}
	\OpR u (x,t) = u_j + \sigma_j \frac{x-x_j}{\Dx},\qquad x\in\cell_j,
\end{equation*}
(see e.g. \cite{godlewski1991hyperbolic,leveque2002finite}). The slopes $\sigma_j$ are selected using an appropriate limiter depending on $u_{j-1},u_j$, and $u_{j+1}$, for example the minmod limiter \cite{leveque2002finite}.
The right and left edge values
\begin{align*}
	u_j^+ &= \lim_{x\to x_\jphf-} \OpR u (x,t) = u_j + \frac{1}{2}\sigma_j\qquad \text{and} \qquad u_j^- = \lim_{x\to x_\jmhf+} \OpR u (x,t) = u_j - \frac{1}{2}\sigma_j
\end{align*}
in the cell $\cell_j$ can then be used instead of the cell averages to give the second-order accurate, semi-discrete finite volume method
\begin{equation}
	\frac{\diff u_j}{\diff t} + \frac{g(u_j^+,u_{j+1}^-)-g(u_{j-1}^+,u_j^-)}{\Dx} = 0.
	\label{second-order local numerical scheme}
\end{equation}
In order for the method to be total variation diminishing (TVD) the slopes have to satisfy
\begin{equation*}
	-2 \leq \frac{\sigma_{j+1}-\sigma_j}{u_{j+1}-u_j} \leq 2,
\end{equation*}
see \cite{sweby1984high}. The TVD property is enough to conclude that limits of the scheme, as $\Dx\to 0$, are at least weak solutions of the conservation law \eqref{local claw}, cf.~\cite{lax1960systems}, but so far no proof that any second-order scheme converges towards the entropy solution of the local conservation law~\eqref{local claw} is available in the literature.


\subsection{Outline of this paper}

The rest of this paper is structured as follows.
In \Cref{sec: Regularity} we define the notions of weak and entropy solutions of \eqref{nonlocal model} and prove that if the nonlocal interaction kernel satisfies a certain growth condition near zero, then those two notions coincide. \Cref{sec: Regularity} further contains two regularity results: We show that weak solutions of the nonlocal model can only exhibit stationary shocks and that traveling wave solutions are either stationary or smooth.
In \Cref{sec: A second-order scheme} we then detail the construction of our second-order scheme. To that end, we first consider the numerical approximation of the weighted integral in \eqref{nonlocal model} and then a suitable time discretization.
We note that the procedure developed in \Cref{sec: A second-order scheme} can readily be modified to higher orders.
Next, we prove certain properties of the forward Euler time discretization, such as the discrete maximum principle and the TVD property, which are then used to show that the scheme converges and that the limit is a weak solution to \eqref{nonlocal model} with a Lax--Wendroff-type theorem. This section also includes a novel nonlocal generalization of Harten's lemma (cf.~\cite{harten1983high}) that is interesting in its own regard. In \Cref{sec: Numerical experiments} we present a series of numerical experiments for the second-order scheme: First, we compare it to the first-order scheme presented in \cite{du2017nonlocal,du2017numerical} and then to a second-order scheme for the local conservation law. Further experiments underpin our findings in \Cref{sec: Regularity} with regards to the regularity of solutions of the nonlocal model and demonstrate asymptotical compatibility with the local entropy solution.


\section{Regularity of weak solutions to the nonlocal model}\label{sec: Regularity}
In this section we will show that solutions of the nonlocal model are more regular than solutions of local conservation laws. To this end, we will first prove that for a certain class of nonlocal interaction kernels weak solutions of the nonlocal model are in fact entropy solutions and hence unique.
Further, we use a Rankine--Hugoniot-type argument to show that any discontinuities present in a weak solution of the nonlocal model must necessarily be stationary.
Lastly, we will show that, for a certain class of nonlocal interaction kernels, traveling wave solutions are smooth.

Throughout this paper we will consider nonlocal interaction kernels $\omega_\delta\colon\R\to\R$ satisfying
\begin{equation*}
	\omega_\delta \geq 0, \qquad \supp \omega_\delta \subseteq [0,\delta],\qquad\text{and}\qquad \int_0^\delta \omega_\delta(h)\diff h=1
\end{equation*}
and numerical fluxes $g\colon\R\times\R\to \R$ which are consistent with a flux $f$, monotone, and Lipschitz continuous, i.e.,
\begin{equation}
\begin{split}
	g(u,u) = f(u),\qquad \partial_1 g\geq 0,\qquad \partial_2 g\leq 0,\\
	\text{and}\qquad |g(u_1,v_1) - g(u_2,v_2)|\leq C(|u_1-u_2|+|v_1-v_2|).
\end{split}
	\label{Assumptions on g}
\end{equation}

As in the case of local conservation laws, we can define a notion of weak solutions for the nonlocal model.
\begin{definition}[Weak solution]
	A function $u\in\Linfty(\R\times(0,T))$ is a weak solution of the nonlocal conservation law \eqref{nonlocal model} if
	\begin{equation*}
		\int_0^T \int_\R u \frac{\partial \phi}{\partial t} \diff x\diff t + \int_\R u_0(x)\phi(x,0)\diff x + \int_0^T\int_\R \int_0^\delta \frac{\tau_h\phi-\phi}{h} g(u,\tau_h u) \omega_\delta(h) \diff h\diff x\diff t = 0
	\end{equation*}
	for all $\phi\in\mathcal{C}_c^\infty(\R\times[0,T))$.
\end{definition}
Furthermore, we will consider entropy solutions in the sense of Kru\v{z}kov as introduced by Du, Huang, and LeFloch \cite{du2017nonlocal}.
\begin{definition}[Entropy solution]\label{def: entropy solution}
	A function $u\in\Linfty(\R\times(0,T))$ is an entropy solution of the nonlocal conservation law \eqref{nonlocal model} if for all $c\in\R$
	\begin{multline*}
		\int_0^T\int_\R |u-c|\frac{\partial\phi}{\partial t}\diff x\diff t + \int_\R |u_0(x)-c|\phi(x,0)\diff x\\
		+ \int_0^T\int_\R\int_0^\delta \frac{\tau_h\phi - \phi}{h}q(u,\tau_h u;c)\omega_\delta (h)\diff h\diff x\diff t \geq 0
	\end{multline*}
	for all nonnegative $\phi\in\mathcal{C}^\infty_c(\R\times[0,T))$. Here $q$ is the nonlocal entropy flux corresponding to the entropy $\eta(u,c) = |u-c|$, defined as\footnote{Note that the second line of \Cref{eq: numerical entropy flux} is not identical to the corresponding equation in \cite[p. 2470]{du2017nonlocal}, which we believe to be a misprint.}
	\begin{align}
		q(a,b;c) &= g(a\vee c,b\vee c) - g(a\wedge c, b\wedge c) \notag\\
		&= \sign(a-c)\sign(b-c)\bigg( \frac{\sign(a-c)+\sign(b-c)}{2}(g(a,b) - g(c,c)) \label{eq: numerical entropy flux}\\
		&\mathrel{\phantom{=}}\phantom{\sign(a-c)\sign(b-c)\bigg(} + \frac{\sign(a-c)-\sign(b-c)}{2}(g(c,b) - g(a,c))\bigg) \notag
	\end{align}
\end{definition}
In \cite{du2017nonlocal}, Du, Huang, and LeFloch were able to show uniqueness of entropy solutions of \eqref{nonlocal model} using Kru\v{z}kov techniques.

\subsection{Uniqueness of weak solutions}
We will now show that if the nonlocal interaction kernel $\omega_\delta$ satisfies
\begin{equation}
	\int_0^\delta \frac{\omega_\delta(h)}{h}\diff h <\infty
	\label{growth condition on omega}
\end{equation}
then any weak solution of \eqref{nonlocal model} is in fact an entropy solution. In particular, this implies that weak solutions are unique and that the second-order scheme we construct in this paper converges towards the unique entropy solution.
Note that the condition \eqref{growth condition on omega} roughly says that $\omega_\delta$ behaves like $h^\alpha$ near zero for some $\alpha>0$. Heuristically speaking, interaction kernels satisfying \eqref{growth condition on omega} place greater weights on long-range interactions than on short-range interactions.
\begin{theorem}\label{theorem 1}
	Assume that the nonlocal interaction kernel $\omega_\delta$ satisfies \eqref{growth condition on omega}
	.
	Then any weak solution $u$ of \eqref{nonlocal model} is an entropy solution. In particular, this implies that weak solutions are unique.
\end{theorem}
\begin{proof}
Let $u$ be a weak solution.
Note that because of the assumption \eqref{growth condition on omega} 
we can rewrite
\begin{equation}
		\int_\R \int_0^\delta \frac{\tau_h\phi-\phi}{h} g(u,\tau_h u) \omega_\delta(h)\diff h\diff x
		=-\int_\R \int_0^\delta \frac{g(u,\tau_h u) - g(\tau_{-h}u,u)}{h} \omega_\delta(h)\phi\diff h\diff x.
		\label{eqn: transfer of difference operator}
\end{equation}
Using the Lipschitz continuity of $g$ and the fact that $u$ is bounded in $\mathrm{L}^\infty$ (see~\cite[Eq. 4.8]{du2017nonlocal}), we get for every $\phi\in\mathcal{C}^\infty_c(\R\times (0,T))$
\begin{align*}
	\int_0^T \int_\R u \frac{\partial\phi}{\partial t} \diff x \diff t
	&= \int_0^T\int_\R\int_0^\delta \frac{g(u,\tau_h u) - g(\tau_{-h}u,u)}{h} \omega_\delta (h)\diff h \phi \diff x\diff t\\
	&\leq C\int_0^T \int_\R \int_0^\delta (|u(x+h)-u(x)|+|u(x)-u(x-h)|)\frac{\omega_\delta(h)}{h}\diff h |\phi(x,t)| \diff x\diff t\\
	&\leq C \|u\|_{\mathrm{L}^\infty((0,T)\times\R)}\int_0^T \int_\R \int_0^\delta \frac{\omega_\delta(h)}{h}\diff h |\phi(x,t)| \diff x\diff t\\
	&= C \|u_0\|_{\mathrm{L}^\infty(\R)}\|\phi\|_{\mathrm{L}^1((0,T)\times\R)} \int_0^\delta \frac{\omega_\delta(h)}{h}\diff h.
\end{align*}
Because of the assumption $\int_0^\delta \frac{\omega_\delta(h)}{h}\diff h <\infty$, we therefore have
\begin{equation*}
	\int_0^T \int_\R u \frac{\partial\phi}{\partial t} \diff x\diff t \leq C \|\phi\|_{\mathrm{L}^1((0,T)\times\R)}
\end{equation*}
which implies that $u(x,\cdot)$ is Lipschitz continuous for almost every~$x\in\R$ (see \Cref{prop:lipschitz} in \Cref{app:lipschitz}) and the weak solution $u$ satisfies~\eqref{nonlocal model} pointwise almost everywhere. If we then multiply~\eqref{nonlocal model} by $\sign(u-c)$ and use the chain rule to simplify $\sign(u-c) \frac{\partial u}{\partial t} = \frac{\partial}{\partial t}|u-c|$ we get
\begin{equation*}
	\frac{\partial}{\partial t}|u-c| + \int_0^\delta \sign(u-c)(g(u,\tau_h u) - g(\tau_{-h}u,u)) \frac{\omega_\delta(h)}{h}\diff h = 0
\end{equation*}
in the distributional sense. In contrast to that, $u$ is an entropy solution in the sense of \Cref{def: entropy solution} if
\begin{equation*}
	\frac{\partial}{\partial t}|u-c| + \int_0^\delta (q(u,\tau_h u;c) - q(\tau_{-h}u,u;c)) \frac{\omega_\delta(h)}{h}\diff h \leq 0
\end{equation*}
holds in the distributional sense. Thus, it remains to show that
\begin{equation*}
		\sign(u-c)(g(u,\tau_h u) - g(\tau_{-h}u,u)) \geq q(u,\tau_h u;c) - q(\tau_{-h}u,u;c).
\end{equation*}
This can be easily verified case by case. If $u\geq c$ then
\begin{align*}
	\sign(u-c)(g(u,\tau_h u) - &g(\tau_{-h}u,u)) - q(u,\tau_h u;c) + q(\tau_{-h}u,u;c)\\
	={} &g(u,\tau_h u) - g(\tau_{-h}u,u) - g(u,\tau_h u\vee c) + g(c,\tau_h u\wedge c)\\
	&{}+ g(\tau_{-h}u\vee c,u) - g(\tau_{-h}u\wedge c, c)\\
	={}&(g(u,\tau_h u) - g(u,\tau_h u\vee c)) + (g(\tau_{-h} u\vee c,u) - g(\tau_{-h}u,u))\\
	&{}+ (g(c,\tau_h u\wedge c) - g(c,c)) + (g(c,c) - g(\tau_{-h}u\wedge c,c)).
\end{align*}
Since $g$ is a monotone flux function, meaning $g$ is monotonically increasing in the first entry and monotonically decreasing in the second, all four terms in parentheses in the preceding line are nonnegative.
The case $u<c$ can be analyzed in the same way.
Therefore, $u$ is an entropy solution and hence, by \cite[Thm. 2.3]{du2017nonlocal}, unique.
\end{proof}

\subsection{Stationarity of discontinuities of weak solutions}
We will now show that any discontinuity in a weak solution of \eqref{nonlocal model} must be stationary. To that end we will employ a Rankine--Hugoniot-type argument (see e.g.~\cite[pp.~8--9]{holden2015front}). We will use the ``integration by parts'' lemma given in Appendix \ref{app: integration by parts}.

\begin{theorem}\label{thm: stationarity of discontinuities of weak solutions}
	Let $u$ be a weak solution of the nonlocal model~\eqref{nonlocal model} which is piecewise $\mathcal{C}^1$ with an isolated discontinuity that moves along a rectifiable curve $\Gamma=\{(x(t),t)\,:\,t\in I\}$ for some interval $I\subset [0,\infty)$. Then $x'\equiv0$; in other words, the discontinuity is stationary.
\end{theorem}

\begin{proof}
Let $D$ be a neighborhood of the point $(x(t_0),t_0)\in\Gamma$ for some fixed $t_0>0$ such that  $\min_t\{ \max_{x,y}\dist( (x,t),(y,t)): (x,t),(y,t)\in D\}>2\delta$. Then $u$ is $\mathcal{C}^1$ inside $D$ except on $\Gamma\cap D$.
Let $D_1,D_2\coloneqq\{(x,t)\in D : x\lessgtr x(t)\}$, and for $\eps>0$ sufficiently small let
\begin{equation*}
	D_i^\eps\coloneqq\big\{ (x,t)\in D_i : \dist\big((x,t),\Gamma\big)>\eps \big\},
\end{equation*}
for $i=1,2$. For $\tilde{\eps}>0$, we define the test function $\phi(x,t) = \psi(t)\varphi_{\eps,\tilde{\eps}}(x,t)$ where $\psi\in\mathcal{C}^1_c(J)$ for $J\coloneqq \{t\in I:(x(t),t)\in\Gamma\cap D\}$, and where $\varphi_{\eps,\tilde{\eps}}\in\mathcal{C}^1(D)$ satisfies $0\leq\varphi_{\eps,\tilde\eps}\leq 1$ and
\begin{equation*}
	\varphi_{\eps,\tilde{\eps}}(x,t) = \begin{cases}
		1& \text{if }|x-x(t)|<\eps,\\
		0& \text{if }|x-x(t)|>\eps+\tilde{\eps}.
	\end{cases}
\end{equation*}
Note that by construction we have
$$\supp \phi = \big\{(x,t)\in D\ :\ t\in \supp \psi \text{ and } x\in(x(t)-(\eps+\tilde{\eps}),x(t)+\eps+\tilde{\eps})\big\}$$
and due to the definition of $D_i^\eps$, for $\eps$ and $\tilde{\eps}$ sufficiently small, we have in particular $\supp \phi \cap \big(D_1^\eps\setminus(\tau_{-h}D_1^\eps)\big) = \supp \phi \cap \big((\tau_{-h} D_2^\eps)\setminus D_2^\eps\big) = \emptyset$.

Since $u$ is a weak solution, we have
\begin{align}
 \begin{split}
 	0 &= \iint_D \left( u\frac{\partial \phi}{\partial t} + \int_0^\delta \frac{\tau_h\phi-\phi}{h} g(u,\tau_h u)\omega_\delta(h)\diff h \right)\diff x\diff t\\
 	&= \lim_{\eps\to 0} \iint_{D_1^\eps\cup D_2^\eps} \left( u\frac{\partial \phi}{\partial t} + \int_0^\delta \frac{\tau_h\phi-\phi}{h} g(u,\tau_h u)\omega_\delta(h)\diff h \right)\diff x\diff t.
 \end{split}
 \label{Rankine-Hugoniot: weak formulation in D}
 \end{align}
 Using \Cref{lem: Integration by parts} to integrate by parts we find
 \begin{align}
 \begin{split}
 	\iint_{D_1^\eps} \biggl(u\frac{\partial\phi}{\partial t} & + \int_0^\delta g(u,\tau_hu)\frac{\tau_h\phi-\phi}{h}\omega_\delta(h)\diff h\biggr) \diff x \diff t \\
    &= -\int_J\int_{x(t)-(\eps+\tilde{\eps})}^{x(t)-\eps}\phi\biggl(\frac{\partial u}{\partial t} + \int_0^\delta \frac{g(u,\tau_h u) - g(\tau_{-h}u,u)}{h}\omega_\delta(h)\diff h\biggr) \diff x \diff t\\
    &\mathrel{\hphantom{=}} +\iint_{\partial D_1^\eps}u\phi n_t^1\diff S + \int_0^\delta\int_J\int_{x(t)-(\eps+\tilde{\eps})}^{x(t)+\eps+\tilde{\eps}}\phi g(\tau_{-h}u,u)\frac{\tau_{-h}\ind_{D_1^\eps}-\ind_{D_1^\eps}}{h}\omega_\delta(h)\diff x\diff t\diff h
\label{eq: D 1 eps}
\end{split}
\end{align}
where $n_t^1$ is the $t$-component of the outward pointing normal to $\partial D_1^\eps$.
We have
\begin{equation*}
	\frac{\tau_{-h}\ind_{D_1^\eps}-\ind_{D_1^\eps}}{h} = \begin{cases}
		-\frac{1}{h}& \text{if }(x,t)\in D_1^\eps\setminus(\tau_{-h}D_1^\eps),\\
		\frac{1}{h}& \text{if }(x,t) \in (\tau_{-h}D_1^\eps)\setminus D_1^\eps, \\
		0 & \text{else.}
	\end{cases}
\end{equation*}
Since $\supp \phi\cap \big(D_1^\eps\setminus(\tau_{-h}D_1^\eps)\big)=\emptyset$ we have
\begin{multline*}
	\int_0^\delta\int_J\int_{x(t)-(\eps+\tilde{\eps})}^{x(t)+\eps+\tilde{\eps}}\phi g(\tau_{-h}u,u)\frac{\tau_{-h}\ind_{D_1^\eps}-\ind_{D_1^\eps}}{h}\omega_\delta(h)\diff x\diff t\diff h\\
	= \int_0^\delta\int_J\int_{x(t)-\eps}^{x(t)-\eps+h} g(\tau_{-h}u,u) \frac{\omega_\delta(h)}{h}\diff x\diff t\diff h
\end{multline*}
which is independent of $\tilde{\eps}$. Thus, sending $\tilde{\eps}\to 0$ in~\eqref{eq: D 1 eps} and using the dominated convergence theorem, we obtain
\begin{align*}
 	&\iint_{D_1^\eps} \biggl(u\frac{\partial\phi}{\partial t} + \int_0^\delta g(u,\tau_hu)\frac{\tau_h\phi-\phi}{h}\omega_\delta(h)\diff h\biggr) \diff x \diff t \\
    &= \int_{\partial D_1^\eps} u\psi(t) n_t^1\diff S + \int_0^\delta\int_J\int_{x(t)-\eps}^{x(t)-\eps+h}g(\tau_{-h}u,u)\frac{\omega_\delta(h)}{h}\diff x\diff t\diff h.
\end{align*}
Similarly, we find
\begin{align*}
 	&\iint_{D_2^\eps} \biggl(u\frac{\partial\phi}{\partial t} + \int_0^\delta g(u,\tau_hu)\frac{\tau_h\phi-\phi}{h}\omega_\delta(h)\diff h\biggr) \diff x \diff t \\
    &= \int_{\partial D_2^\eps} u\psi(t) n_t^2\diff S - \int_0^\delta\int_J\int_{x(t)+\eps}^{x(t)+\eps+h}g(\tau_{-h}u,u)\frac{\omega_\delta(h)}{h}\diff x\diff t\diff h.
\end{align*}
where $n_t^2$ is the $t$-component of the outward pointing normal to $\partial D_2^\eps$.
Going back to~\eqref{Rankine-Hugoniot: weak formulation in D} and using $n_t^1 = -n_t^2$ we obtain
\begin{align*}
	0 &= \lim_{\eps\to 0} \iint_{D_1^\eps\cup D_2^\eps} \left( u\frac{\partial \phi}{\partial t} + \int_0^\delta \frac{\tau_h\phi-\phi}{h} g(u,\tau_h u)\omega_\delta(h)\diff h \right)\diff x\diff t\\
	&= \lim_{\eps\to 0} \left( \int_{\partial D_1^\eps} u\psi(t)n_t^1\diff S - \int_{\partial D_2^\eps} u\psi(t)n_t^1\diff S \right)\\
	&\mathrel{\hphantom{=}}+ \lim_{\eps\to 0} \int_0^\delta\bigg(\int_J\int_{x(t)-\eps}^{x(t)-\eps+h} g(\tau_{-h}u,u)\frac{\omega_\delta(h)}{h}\diff x\diff t\diff h\\
	&\mathrel{\hphantom{=+ \lim_{\eps\to 0} \int_0^\delta\bigg(}}- \int_J\int_{x(t)+\eps}^{x(t)+\eps+h} g(\tau_{-h}u,u)\frac{\omega_\delta(h)}{h}\diff x\diff t\bigg)\diff h\\
	&= -\int_J \psi(t) (u_--u_+)x'(t)\diff t
\end{align*}
where $u_\pm(t)\coloneqq \lim_{x\to x(t)\pm0}u(x,t)$ are the right and left limits of $u(\cdot,t)$ as $x\to x(t)$, respectively. Since we chose the test function $\psi$ arbitrarily and $u_-\neq u_+$ (as $u$ is discontinuous), we conclude that $x'(t)=0$ for all $t$. Thus, the discontinuity is stationary.
\end{proof}

\subsection{Regularity of traveling wave solutions}
For traveling wave solutions we can prove the following stronger result.
\begin{theorem}
	Assume that the nonlocal interaction kernel $\omega_\delta$ satisfies $\int_0^\delta \frac{\omega_\delta(h)}{h}\diff h<\infty$. Then every traveling wave solution, i.e., any weak solution of the form $u(x,t)=v(x-ct)$ for some $c\in\R$, is either stationary or lies in $\mathcal{C}^{k+1,1}$, where $k\geq 0$ is the largest integer such that $g\in\mathcal{C}^{k,1}$.
\end{theorem}
\begin{proof}
From the proof of \Cref{theorem 1} we recall that if $\int_0^\delta \frac{\omega_\delta(h)}{h}\diff h <\infty$, then any weak solution satisfies \eqref{nonlocal model} pointwise almost everywhere. In the case of a traveling wave solution this means
\begin{equation}
	c v'(\xi) = \int_0^\delta \frac{g(v,\tau_h v) - g(\tau_{-h}v,v)}{h} \omega_\delta(h)\diff h
	\label{eq: travelling wave}
\end{equation}
where we have set $\xi=x-ct$.
If $c=0$ then $u$ is stationary. On the other hand, if $c\neq 0$ we can use equation \eqref{eq: travelling wave} together with the Lipschitz continuity of $g$ and $\int_0^\delta \frac{\omega_\delta(h)}{h}\diff h<\infty$ to find that
\begin{equation*}
	|v'(\xi)| \leq \frac{1}{|c|}\int_0^\delta |g(v,\tau_h v) - g(\tau_{-h}v,v)| \frac{\omega_\delta(h)}{h}\diff h \leq C\|v\|_\infty \int_0^\delta \frac{\omega_\delta(h)}{h}\diff h \leq C
\end{equation*}
which implies that $v'\in\Linfty(\R)$. Thus, $v$ is Lipschitz continuous. Again going back to \eqref{eq: travelling wave}, the right-hand side, and hence also $v'$, must be Lipschitz continuous, and therefore $v\in\mathcal{C}^{1,1}(\R)$. Continuing this bootstrap argument, we see that if $c\neq 0$ then $v\in\mathcal{C}^{k+1,1}$ whenever $g\in\mathcal{C}^{k,1}$.
\end{proof}
\begin{rem}
	In \cite{du2016analysis} the authors considered a parabolic equation with nonlocality in time and reported spatial smoothing properties. Similarly, in our setting we expected to gain regularity in time as a consequence of the nonlocality in space.
\end{rem}

\section{A second-order scheme for the nonlocal model}\label{sec: A second-order scheme}
We now propose a second-order scheme for the nonlocal model \eqref{nonlocal model}. Let $\Dx$ and $\Dt$ denote the spatial and temporal grid size and denote the spatial and temporal cells by $\cell_j = (x_\jmhf,x_\jphf)$ and $\timecell^n = (t^n,t^{n+1})$ and the grid points by $x_j= j\Dx$ and $t^n = n\Dt$.
Here, we assumed a uniform grid for simplicity, but note that the scheme can be generalized to non-uniform grids.
Let further $u_j^n$ denote the numerical solution at the point $(x_j,t^n)$. In the following, we will denote $N=T/\Dt$ and $\lambda=\Dt/\Dx$.

\subsection{Discretization of the integral term}
We will first approximate the integral term in \eqref{nonlocal model}. In \cite{du2017nonlocal} a discretization based on the endpoint rule was proposed and the numerical examples presented in \cite{du2017numerical} underline its first-order accuracy. In order to design a second-order accurate method we will employ an approximation based on the trapezoidal rule.
Let $r = \lfloor \frac{\delta}{\Dx}\rfloor$ and denote $\I_k = ((k-1)\Dx,k\Dx)$, $k\in\N$.
Following \cite{tian2013analysis}, for functions $G$ with bounded second derivative on $[0,\delta]$ we have
\begin{equation}
	\int_0^{\delta}G(h)\omega_{\delta}(h)\diff h = \sum_{k=1}^{r+ 1} \int_{\I_k} \widehat{G}(h)\omega_\delta(h) \diff h + \mathcal{O}(\Dx^2),\label{eqn: integral approximation}
\end{equation}
where
\begin{equation*}
	\widehat{G}(h) = \sum_{i=0}^{r+ 1} G(i\Dx)\Phi_i(h)
\end{equation*}
is a piecewise linear approximation to $G$ utilizing the standard continuous piecewise linear hat functions
\begin{equation*}
	\Phi_i (h) = \begin{cases}
		\frac{h-(i-1)\Dx}{\Dx} & \text{if }h\in \I_i,\\
		\frac{(i+1)\Dx-h}{\Dx} & \text{if }h\in \I_{i+1},\\
		0 & \text{otherwise.}
	\end{cases}
\end{equation*}
Taking into account the support of $\Phi_i$ we can calculate
\begin{align*}
		\sum_{k=1}^{r+1} \int_{\mathcal{I}_k} \widehat{G}(h) \omega_\delta(h) \diff h
		&= \sum_{k=1}^{r+1} \int_{\mathcal{I}_k} \sum_{i=0}^{r+1} G(i\Dx)\Phi_i(h) \omega_\delta(h) \diff h\\
		&= G(0)\int_{\mathcal{I}_1}\Phi_0(h)\omega_\delta(h)\diff h +\sum_{k=1}^{r+1} G(k\Dx) \int_{\mathcal{I}_k\cup \mathcal{I}_{k+1}} \Phi_k(h) \omega_\delta(h) \diff h
	\end{align*}
In order to approximate the integrand in \eqref{nonlocal model} we set
\begin{equation*}
	G(h) = \frac{g(u(x_j), u(x_j+h))-g(u(x_j-h),u(x_j))}{h}.
\end{equation*}
We will approximate $G(0) \coloneq \lim_{h\to 0} G(h)$ by
\begin{equation*}
	G(0) \approx \frac{g(u_j^+,u_{j+1}^-) - g(u_{j-1}^+,u_j^-)}{\Dx},
\end{equation*}
where
\begin{equation*}
	u_j^+ = u_j + \frac{1}{2}\sigma_j \qquad \text{and} \qquad u_j^- = u_j - \frac{1}{2}\sigma_j
\end{equation*}
are the values of a piecewise linear reconstruction at the cell interfaces (see e.g. \cite{godlewski1991hyperbolic,leveque2002finite}). Here, $\sigma_j$ is an approximation to the slope in the interval $\cell_j$. In the following, we will assume that the approximate slopes $\sigma_j$ satisfy
\begin{equation}
	-2\leq \frac{\sigma_{j+1} - \sigma_j}{u_{j+1}-u_j} \leq 2 \qquad\text{for all }j,
	\label{Assumption on the slopes}
\end{equation}
which is the well-known TVD region for slope-limiter methods, see \cite{sweby1984high}.
We obtain the following approximation of the integral in \eqref{nonlocal model}
at the points $x=x_j$:
\begin{equation*}
	 \frac{g(u_j^+,u_{j+1}^-) - g(u_{j-1}^+,u_j^-)}{\Dx} W_0 + \sum_{k=1}^{r+1} \frac{g(u_j,u_{j+k}) - g(u_{j-k},u_j)}{k\Dx} W_k,
\end{equation*}
where
\begin{equation}
\begin{aligned}
	W_0 &= \int_{\I_1} \Phi_0(h)\omega_\delta(h)\diff h &&\text{and}\\
	W_k &= \int_{\I_k\cup \I_{k+1}} \Phi_k(h)\omega_\delta(h)\diff h, &&\text{for }k=1,\ldots,r+1.
\end{aligned}
\label{second-order weights}
\end{equation}
See \Cref{fig: weights} for an illustration of the weights.
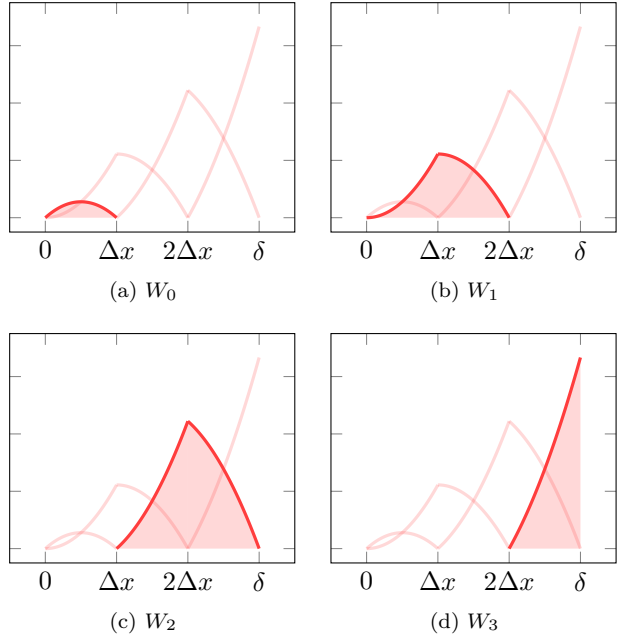
\begin{figure}[t]
	\centering
	\subfloat[$W_0$]{
	\begin{tikzpicture}
		\begin{axis}[xmin=-0.5,xmax=3.5,xtick={0,1,2,3},xticklabels={$0$,$\Dx$,$2\Dx$,$\delta$},yticklabels=\empty,ymin=-0.05,ymax=0.75]
			\addplot+[name path=f,myred, very thick, domain=0:1, mark=none, samples=50] {(1-x)*2*x/(3^2)};

			\addplot+[myred, opacity=0.2, very thick, domain=0:1, mark=none, samples=50] {(x-0)*2*x/(3^2)};
			\addplot+[myred, opacity=0.2, very thick, domain=1:2, mark=none, samples=50] {(2-x)*2*x/(3^2)};

			\addplot+[myred, opacity=0.2, very thick, domain=1:2, mark=none, samples=50] {(x-1)*2*x/(3^2)};
			\addplot+[myred, opacity=0.2, very thick, domain=2:3, mark=none, samples=50] {(3-x)*2*x/(3^2)};

			\addplot[myred, opacity=0.2, very thick, domain=2:3, mark=none, samples=50] {(x-2)*2*x/(3^2)};

			\path[name path=axis] (axis cs:0,0) -- (axis cs:1,0);

   			\addplot [
   			    thick,
   			    color=myred,
   			    fill=myred,
   			    fill opacity=0.2
   			]
   			fill between[
   			    of=f and axis,
   			    soft clip={domain=0:1},
   			];
		\end{axis}
	\end{tikzpicture}
	}
	\subfloat[$W_1$]{
	\begin{tikzpicture}
		\begin{axis}[xmin=-0.5,xmax=3.5,xtick={0,1,2,3},xticklabels={$0$,$\Dx$,$2\Dx$,$\delta$},yticklabels=\empty,ymin=-0.05,ymax=0.75]
			\addplot+[myred, opacity=0.2, very thick, domain=0:1, mark=none, samples=50] {(1-x)*2*x/(3^2)};

			\addplot+[name path=f,myred, very thick, domain=0:1, mark=none, samples=50] {(x-0)*2*x/(3^2)};
			\addplot+[name path=g,myred, very thick, domain=1:2, mark=none, samples=50] {(2-x)*2*x/(3^2)};

			\addplot+[myred, opacity=0.2, very thick, domain=1:2, mark=none, samples=50] {(x-1)*2*x/(3^2)};
			\addplot+[myred, opacity=0.2, very thick, domain=2:3, mark=none, samples=50] {(3-x)*2*x/(3^2)};

			\addplot[myred, opacity=0.2, very thick, domain=2:3, mark=none, samples=50] {(x-2)*2*x/(3^2)};
			\path[name path=axis] (axis cs:0,0) -- (axis cs:1,0);

   			\addplot [
   			    thick,
   			    color=myred,
   			    fill=myred,
   			    fill opacity=0.2
   			]
   			fill between[
   			    of=f and axis,
   			    soft clip={domain=0:1},
   			];
   			\addplot [
   			    thick,
   			    color=myred,
   			    fill=myred,
   			    fill opacity=0.2
   			]
   			fill between[
   			    of=g and axis,
   			    soft clip={domain=1:2},
   			];
		\end{axis}
	\end{tikzpicture}
	}\\
	\subfloat[$W_2$]{
	\begin{tikzpicture}
		\begin{axis}[xmin=-0.5,xmax=3.5,xtick={0,1,2,3},xticklabels={$0$,$\Dx$,$2\Dx$,$\delta$},yticklabels=\empty,ymin=-0.05,ymax=0.75]
			\addplot+[myred, opacity=0.2, very thick, domain=0:1, mark=none, samples=50] {(1-x)*2*x/(3^2)};

			\addplot+[myred, opacity=0.2, very thick, domain=0:1, mark=none, samples=50] {(x-0)*2*x/(3^2)};
			\addplot+[myred, opacity=0.2, very thick, domain=1:2, mark=none, samples=50] {(2-x)*2*x/(3^2)};

			\addplot+[name path=f,myred, very thick, domain=1:2, mark=none, samples=50] {(x-1)*2*x/(3^2)};
			\addplot+[name path=g,myred, very thick, domain=2:3, mark=none, samples=50] {(3-x)*2*x/(3^2)};

			\addplot[myred, opacity=0.2, very thick, domain=2:3, mark=none, samples=50] {(x-2)*2*x/(3^2)};
			\path[name path=axis] (axis cs:0,0) -- (axis cs:3,0);

   			\addplot [
   			    thick,
   			    color=myred,
   			    fill=myred,
   			    fill opacity=0.2
   			]
   			fill between[
   			    of=f and axis,
   			    soft clip={domain=1:2},
   			];
   			\addplot [
   			    thick,
   			    color=myred,
   			    fill=myred,
   			    fill opacity=0.2
   			]
   			fill between[
   			    of=g and axis,
   			    soft clip={domain=2:3},
   			];
		\end{axis}
	\end{tikzpicture}
	}
	\subfloat[$W_3$]{
	\begin{tikzpicture}
		\begin{axis}[xmin=-0.5,xmax=3.5,xtick={0,1,2,3},xticklabels={$0$,$\Dx$,$2\Dx$,$\delta$},yticklabels=\empty,ymin=-0.05,ymax=0.75]
			\addplot+[myred, opacity=0.2, very thick, domain=0:1, mark=none, samples=50] {(1-x)*2*x/(3^2)};

			\addplot+[myred, opacity=0.2, very thick, domain=0:1, mark=none, samples=50] {(x-0)*2*x/(3^2)};
			\addplot+[myred, opacity=0.2, very thick, domain=1:2, mark=none, samples=50] {(2-x)*2*x/(3^2)};

			\addplot+[myred, opacity=0.2, very thick, domain=1:2, mark=none, samples=50] {(x-1)*2*x/(3^2)};
			\addplot+[myred, opacity=0.2, very thick, domain=2:3, mark=none, samples=50] {(3-x)*2*x/(3^2)};

			\addplot[name path=f,myred, very thick, domain=2:3, mark=none, samples=50] {(x-2)*2*x/(3^2)};
			\path[name path=axis] (axis cs:0,0) -- (axis cs:3,0);

   			\addplot [
   			    thick,
   			    color=myred,
   			    fill=myred,
   			    fill opacity=0.2
   			]
   			fill between[
   			    of=f and axis,
   			    soft clip={domain=2:3},
   			];
		\end{axis}
	\end{tikzpicture}
	}
	\caption{Weights for $\delta=3\Dx$ and $\omega_\delta(h)=\frac{2}{\delta^2}h$. The weights correspond to the highlighted areas.}\label{fig: weights}
\end{figure}

\subsection{The numerical scheme}
In order to keep the presentation brief we will introduce the following notation:
\begin{align*}
	g_\jphf &\coloneqq g(u_j^+,u_{j+1}^-), \qquad j\in\Z,\\
	g_{j,j+k} &\coloneqq g(u_j,u_{j+k}), \qquad j\in\Z, k\in\N.
\end{align*}
Then, defining
\begin{equation*}
	\Approx(u)_j \coloneqq \frac{g_\jphf - g_\jmhf}{\Dx} W_0 + \sum_{k=1}^{r+1} \frac{g_{j,j+k} - g_{j-k,j}}{k\Dx} W_k,
\end{equation*}
where $W_0,\ldots,W_{r+1}$ are given by \eqref{second-order weights}
the semi-discrete numerical scheme reads
\begin{equation}
\left\{
\begin{aligned}
	&\frac{\diff u_j}{\diff t} + \Approx(u)_j = 0\\
	&u_j(0) = \frac{1}{\Dx} \int_{\cell_j}u_0(x)\diff x
\end{aligned}
\right. \qquad \text{for }j\in\Z.
\label{semi-discrete num scheme}
\end{equation}
The following proposition shows that by keeping the spatial grid size $\Dx$ fixed and letting $\delta\to 0$, the first equation in \eqref{semi-discrete num scheme} reduces to the second-order scheme \eqref{second-order local numerical scheme} for the local conservation law.
\begin{proposition}\label{prop: vanishing delta}
Let $\Dx>0$ be fixed. Then the nonlocal semi-discrete scheme \eqref{semi-discrete num scheme} converges to the local semi-discrete scheme \eqref{second-order local numerical scheme} as $\delta\to0$.
\end{proposition}
\begin{proof}
	If $\delta<\Dx$ then $r=0$, so we only need to consider the weights $W_0$ and $W_1$ given by \eqref{second-order weights}. If $\delta<\Dx$ and $h\in(0,\delta)$ then
	\begin{align*}
		\Phi_0(h)&=\frac{\Dx-h}{\Dx}\leq 1 &&\text{and} & \Phi_0(h) &=\frac{\Dx-h}{\Dx}\geq\frac{\Dx-\delta}{\Dx}=1-\frac{\delta}{\Dx}.
	\end{align*}
	Since $\omega_\delta$ is nonnegative and integrates to $1$, we can bound
	\begin{equation*}
		 W_0 = \int_0^\delta \Phi_0(h) \omega_\delta(h)\diff h \leq 1\qquad\text{and}\qquad W_0 \geq 1-\frac{\delta}{\Dx}.
	\end{equation*}
	Similarly, we find $0\leq W_1\leq\frac{\delta}{\Dx}$. Thus, taking $\delta\to 0$, we get $W_0=1$ and $W_1=0$ which means that the nonlocal scheme \eqref{semi-discrete num scheme} converges to the local scheme \eqref{second-order local numerical scheme}.
\end{proof}

For simplicity, we will confine the analysis in the following sections to the forward Euler discretization
\begin{equation}
	u_j^{n+1} = u_j^n - \Dt \Approx(u^n)_j
	\label{FE discretization}
\end{equation}
of \eqref{semi-discrete num scheme}. In order to numerically preserve the second-order accuracy, for the numerical experiments in \Cref{sec: Numerical experiments} we will instead use the strong stability preserving Runge--Kutta discretization given by
\begin{equation}
\begin{split}
	u_j^* &= u_j^n - \Dt\Approx(u^n)_j,\\
	u_j^{**} &= u_j^* - \Dt\Approx(u^*)_j,\\
	u_j^{n+1} &= \frac{1}{2}(u_j^n + u_j^{**}),
\end{split}
\label{num scheme}
\end{equation}
see \cite{gottlieb2001strong} for details.

\subsection{Properties of the numerical scheme}

In this section we will prove essential properties of the numerical scheme \eqref{FE discretization}. First, we present two lemmata which are nonlocal modifications of Harten's lemma, see \cite{harten1983high}.

\begin{lemma}\label{Harten's lemma L infty bound}
	Let $W_0,\dots,W_{r+1}$ be nonnegative numbers such that $\sum_{k=0}^{r+1}W_k =1$ and assume that there are numbers $A_\jphf^n, B_\jphf^n, C_{j,j+k}^n,D_{j-k,j}^n\in\R$ (for each $j\in\Z,\ n\in\N_0,\ k=0,\dots,r+1$) satisfying
	\begin{equation}
		A_\jphf^n, B_\jphf^n, C_{j,j+k}^n,D_{j-k,j}^n \geq 0,\qquad (A_\jphf^n + B_\jmhf^n) W_0 + \sum_{k=1}^{r+1}(C_{j,j+k}^n + D_{j-k,j}^n)\frac{1}{k} W_k \leq 1 \label{Assumptions L infty bound}
	\end{equation}
	for all $n,j$. Then solutions computed with  the scheme
	\begin{equation}
		\begin{split}
			u_j^{n+1} &= u_j^n + \left( A_\jphf^n (u_{j+1}^n - u_j^n) - B_\jmhf^n (u_j^n - u_{j-1}^n) \right) W_0\\
			&\mathrel{\phantom{=}}+ \sum_{k=1}^{r+1} \left( C_{j,j+k}^n (u_{j+k}^n - u_j^n)  - D_{j-k,j}^n (u_j^n - u_{j-k}^n) \right) \frac{1}{k} W_k
		\end{split}
		\label{second-order method incremental form I}
	\end{equation}
 enjoy the discrete maximum principle
	\begin{equation*}
		\inf_i u_i^n \leq u_j^{n+1} \leq \sup_i u_i^n\qquad \text{for all }n,j.
	\end{equation*}
\end{lemma}
\begin{proof}
	By rearranging \eqref{second-order method incremental form I} we get
	\begin{align*}
		u_j^{n+1} &= u_j^n + \left( A_\jphf^n (u_{j+1}^n - u_j^n) - B_\jmhf^n (u_j^n - u_{j-1}^n) \right) W_0\\
		&\mathrel{\phantom{=}}{}+ \sum_{k=1}^{r+1} \left( C_{j,j+k}^n (u_{j+k}^n - u_j^n)  - D_{j-k,j}^n (u_j^n - u_{j-k}^n) \right) \frac{1}{k} W_k\\
		&= \left( 1- (A_\jphf^n +B_\jmhf^n)W_0 - \sum_{k=1}^{r+1} (C_{j,j+k}^n + D_{j-k,j}^n)\frac{1}{k}W_k \right) u_j^n\\
		&\mathrel{\phantom{=}}{} +A_\jphf^n W_0 u_{j+1}^n + B_\jmhf^n W_0 u_{j-1}^n + \sum_{k=1}^{r+1} C_{j,j+k}^n \frac{1}{k} W_k u_{j+k}^n + \sum_{k=1}^{r+1} D_{j-k,j}^n \frac{1}{k} W_k u_{j-k}^n
	\end{align*}
	The condition \eqref{Assumptions L infty bound} ensures that $u_j^{n+1}$ is a convex combination of $u_{j-k}^n,\ldots,u_{j+k}^n$, which proves the discrete maximum principle.
\end{proof}
%
\begin{lemma}\label{Harten's lemma TV bound}
	Let $W_k$ be nonnegative numbers such that $\sum_{k=0}^{r + 1} W_k=1$ and assume that there are numbers $A_\jphf^n,B_\jphf^n,E_\jphf^n, F_\jphf^n\in\R$	(for each $j\in\Z,\ n\in\N_0$) satisfying
	\begin{equation}
		A_\jphf^n,B_\jphf^n,E_\jphf^n, F_\jphf^n \geq 0, \qquad A_\jphf^n+B_\jphf^n+E_\jphf^n+F_\jphf^n \leq 1 \qquad\text{for all }n,j.
		\label{Assumptions Harten's lemma}
	\end{equation}
	Then solutions computed with  the scheme
	\begin{equation}
	\begin{split}
		u_j^{n+1} &= u_j^n + \left( A_\jphf^n (u_{j+1}^n - u_j^n) - B_\jmhf^n (u_j^n - u_{j-1}^n) \right) W_0\\
	&\mathrel{\phantom{=}}+ \sum_{k=1}^{r+1} \sum_{l=1}^k \left( E_\jplmhf^n (u_{j+l}^n - u_{j+l-1}^n)  - F_\jmlphf^n (u_{j-l+1}^n - u_{j-l}^n) \right) \frac{1}{k} W_k
	\end{split}
	\label{second-order method incremental form II}
	\end{equation}
 are TVD, i.e., they satisfy
	\begin{equation}
		\sum_j \left|u_{j+1}^{n+1}-u_j^{n+1}\right| \leq \sum_j \left|u_{j+1}^n -u_j^n\right|.
		\label{TVD estimate}
	\end{equation}
\end{lemma}

\begin{proof}
	Using the incremental form \eqref{second-order method incremental form II}, we obtain
	\begin{align*}
		u_{j+1}^{n+1}-u_j^{n+1}
		&= u_{j+1}^n-u_j^n + \Big( A_{j+\thf}^n (u_{j+2}^n - u_{j+1}^n) - A_\jphf^n (u_{j+1}^n - u_j^n)\\
		&\mathrel{\hphantom{=u_{j+1}^n-u_j^n + \Big(}}- B_\jphf^n (u_{j+1}^n - u_j^n) + B_\jmhf^n (u_j^n - u_{j-1}^n) \Big) W_0 \\
		&\mathrel{\phantom{=}} + \sum_{k=1}^{r + 1}\sum_{l=1}^k \bigg(\begin{aligned}[t] &E_{j+l+\hf}^n(u_{j+l+1}^n - u_{j+l}^n) - E_{j+l-\hf}^n(u_{j+l}^n-u_{j+l-1}^n)\\
		&- F_{j-l+\thf}^n(u_{j-l+2}^n-u_{j-l+1}^n) + F_{j-l+\hf}^n(u_{j-l+1}^n-u_{j-l}^n) \bigg) \frac{1}{k} W_k\end{aligned}\\
		&= \left( 1-\left(A_\jphf^n + B_\jphf^n\right) W_0 - \left(E_\jphf^n + F_\jphf^n\right)\left(\sum_{k=1}^{r+1}\frac{1}{k}W_k\right) \right) (u_{j+1}^n - u_j^n) \\
		&\mathrel{\phantom{=}}+ \left( A_{j+\thf}^n (u_{j+2}^n - u_{j+1}^n)+ B_\jmhf^n (u_j^n - u_{j-1}^n) \right) W_0\\
		&\mathrel{\phantom{=}} + \sum_{k=1}^{r + 1} \left( E_{j+k+\hf}^n(u_{j+k+1}^n - u_{j+k}^n)  + F_{j-k+\hf}^n(u_{j-k+1}^n-u_{j-k}^n) \right) \frac{1}{k} W_k\\
	\end{align*}
	Hence, taking absolute values and using
	\begin{equation*}
		W_0\leq 1,\qquad \sum_{k=1}^{r+1} \frac{1}{k} W_k \leq 1,
	\end{equation*}
	and the assumptions \eqref{Assumptions Harten's lemma} yields
	\begin{align*}
		\left|u_{j+1}^{n+1}-u_j^{n+1}\right| &\leq \left(1 -(A_\jphf^n + B_\jphf^n)W_0 -(E_\jphf^n + F_\jphf^n)\left(\sum_{k=1}^{r+1}\frac{1}{k}W_k\right)\right)\left|u_{j+1}^n-u_j^n\right|\\
		&\mathrel{\phantom{\leq}}+ \left( A_{j+\thf}^n\left|u_{j+2}^n - u_{j+1}^n\right| + B_\jmhf^n\left|u_j^n - u_{j-1}^n\right|\right) W_0 \\
		&\mathrel{\phantom{\leq}}+ \sum_{k=1}^{r + 1} \left( E_\jpkphf^n \left|u_{j+k+1}^n-u_{j+k}^n\right| + F_\jmkphf^n \left|u_{j-k+1}^n-u_{j-k}^n\right| \right) \frac{1}{k} W_k.
	\end{align*}
	Summing over $j$ and using
	\begin{equation*}
		\sum_j \sum_{k=1}^{r + 1}E_\jpkphf^n \left|u_{j+k+1}^n-u_{j+k}^n\right|\frac{1}{k} W_k = \left(\sum_{k=1}^{r + 1} \frac{1}{k}W_k\right) \sum_j E_\jphf^n \left|u_{j+1}^n-u_{j}^n\right|
	\end{equation*}
	and a similar calculation for the term involving $F_\jmkphf^n$,
	by identifying equal terms, we obtain the TVD estimate \eqref{TVD estimate}.
\end{proof}
With the help of these two lemmata we can show that the scheme \eqref{FE discretization} for the nonlocal model satisfies the maximum principle and the TVD property under a suitable CFL condition.
\begin{proposition}[Properties of the numerical scheme]\label{lem: properties of the scheme}
	Assume that $u_0\in \Lone(\R)\cap\BV(\R)$ and that the following CFL condition holds:
		\begin{equation}
			\frac{\Dt}{\Dx}\big(\partial_1 g(u,v) - \partial_2 g(w,v)\big) \leq 1
			\label{CFL condition}
		\end{equation}
		for all $u,v,w\in\R$. Then the scheme \eqref{FE discretization} satisfies the following properties:
	\begin{enumerate}[(i)]
		\item It is conservative.
		\item \label{item maximum principle} It enjoys the discrete maximum principle
		\begin{equation}
			\inf_i u_i^n \leq u_j^{n+1} \leq \sup_i u_i^n\qquad\text{for all }n,j.
			\label{discrete maximum principle}
		\end{equation}
		\item It satisfies the TVD property
		\begin{equation*}
			\sum_j \left|u_{j+1}^{n+1}-u_j^{n+1}\right| \leq \sum_j \left|u_{j+1}^n -u_j^n\right|.
		\end{equation*}
		\item It is uniformly $\Lone$-continuous in time as $\Dt\to 0$. More precisely,
		\begin{equation*}
			\int_\R \left| u_\Dx(x,t) - u_\Dx(x,s) \right| \diff x \leq C |t-s| + \mathcal{O}(\Dt).
		\end{equation*}
	\end{enumerate}
\end{proposition}
\begin{proof}
	\begin{enumerate}[(i)]
		\item This is straightforward from the definition.
		\item Defining
		\begin{align*}
			A_\jphf^n &\coloneq -\lambda \frac{g^n_\jphf - g(u_j^{n,+},u_j^{n,-})}{u_{j+1}^n - u_j^n} &
			B_\jmhf^n &\coloneq \lambda \frac{g(u_j^{n,+},u_j^{n,-}) - g^n_\jmhf}{u_j^n - u_{j-1}^n}\\
			C_{j,j+k}^n &\coloneq -\lambda \frac{g^n_{j,j+k} - g^n_{j,j}}{u_{j+k}^n - u_{j}^n} & D_{j-k,j}^n &\coloneq \lambda \frac{g^n_{j,j} - g^n_{j-k,j}}{u_{j}^n - u_{j-k}^n}
		\end{align*}
		we can rewrite the scheme \eqref{FE discretization} in the form \eqref{second-order method incremental form I}.
		It remains to verify that the conditions \eqref{Assumptions L infty bound} of \Cref{Harten's lemma L infty bound} are satisfied in order to conclude the discrete maximum principle.
		Using the mean value theorem, we can calculate
		\begin{align*}
			A_\jphf^n &= - \lambda\partial_2 g(u_j^{n,+},\xi) \frac{u_{j+1}^{n,-} - u_j^{n,-}}{u_{j+1}^n - u_j^n}\\
			&=-\lambda\partial_2 g(u_j^{n,+},\xi) \frac{u_{j+1}^n-\frac{1}{2}\sigma_{j+1}^n-u_j^n +\frac{1}{2}\sigma_j^n}{u_{j+1}^n - u_j^n}\\
			&=-\lambda\partial_2 g(u_j^{n,+},\xi) \left( 1-\frac{1}{2}\frac{\sigma_{j+1}^n-\sigma_j^n}{u_{j+1}^n - u_j^n}\right)
		\intertext{for some $\xi$ between $u_{j+1}^{n,-}$ and $u_j^{n,-}$. Similarly,}
			B_\jphf^n &=\lambda\partial_1 g(\zeta,u_j^{n,-}) \left( 1 + \frac{1}{2}\frac{\sigma_{j+1}^n-\sigma_{j}^n}{u_{j+1}^n - u_j^n}\right).
		\intertext{On the other hand,}
		C_{j,j+k}^n &=-\lambda\partial_2 g\left(u_j^n,\widetilde{\xi}\right) \frac{u_{j+k}^n-u_j^n}{u_{j+k}^n - u_j^n} = -\lambda\partial_2 g\left(\widetilde{\xi}\right)
		\intertext{and analogously}
		D_{j-k,j}^n &= \lambda\partial_1 g\left(\widetilde{\zeta},u_j^n\right).
		\end{align*}
		Since the approximated slopes $\sigma_j$ satisfy the TVD requirement \eqref{Assumption on the slopes}
		and $g$ satisfies the monotonicity assumption \eqref{Assumptions on g}, the CFL condition \eqref{CFL condition} ensures that \eqref{Assumptions L infty bound} holds.
		\item With the help of the telescoping sum
		\begin{equation*}
			g^n_{j,j+k} - g^n_{j-k,j} = \sum_{l=1}^k \big((g^n_{j,j+l} - g^n_{j,j+l-1}) + (g^n_{j-l+1,j} - g^n_{j-l,j})\big),
		\end{equation*}
		defining $A_\jphf^n$ and $B_\jmhf^n$ as in (ii) and
		\begin{align*}
			E_\jplmhf^n \coloneq -\lambda \frac{g^n_{j,j+l} - g^n_{j,j+l-1}}{u_{j+l}^n - u_{j+l-1}^n}, \qquad
			F_\jmlphf^n \coloneq \lambda\frac{g^n_{j-l+1,j} - g^n_{j-l,j}}{u_{j-l+1}^n - u_{j-l}^n},
		\end{align*}
		we can rewrite the scheme \eqref{FE discretization} in the form \eqref{second-order method incremental form II}. A similar calculation as in (ii) shows that the assumptions \eqref{Assumptions Harten's lemma} of \Cref{Harten's lemma TV bound} are satisfied.
		\item It remains to show the $\Lone$ continuity in time as $\Dt\to 0$. First note that
	\begin{align*}
		|u_j^{n+1} -u_j^n|
		&= \lambda \left|(g^n_\jphf - g^n_\jmhf)W_0 + \sum_{k=1}^{r+1} (g^n_{j,j+k} - g^n_{j-k,j})\frac{1}{k}W_k\right|\\
		&\leq \lambda \left|g^n_\jphf - g^n_\jmhf\right|W_0 + \sum_{k=1}^{r+1}\sum_{l=1}^k \Big(\left|g^n_{j,j+l}-g^n_{j,j+l-1}\right| +\left|g^n_{j-l+1,j} - g^n_{j-l,j}\right|\Big)\frac{1}{k}W_k
	\end{align*}
	With the Lipschitz continuity and \eqref{Assumption on the slopes} we find
	\begin{equation*}
		\left|g^n_{j,j+l}-g^n_{j,j+l-1}\right|  \leq C \left| u_{j+l}^n - u_{j+l-1}^n \right|
	\end{equation*}
	as well as
	\begin{align*}
		\left|g^n_\jphf - g^n_\jmhf \right| &\leq \left|g(u_j^{n,+},u_{j+1}^{n,-}) -g(u_j^{n,+},u_j^{n,-})\right| + \left| g(u_j^{n,+},u_j^{n,-}) - g(u_{j-1}^{n,+},u_j^{n,-})\right|\\
		&\leq C \left(\left| u_{j+1}^n - u_j^n\right| + \frac{1}{2}\left|\sigma_{j+1}^n - \sigma_j^n\right| + \left|u_j^n-u_{j-1}^n\right| + \frac{1}{2}\left|\sigma_j^n - \sigma_{j-1}^n\right| \right)\\
		&\leq C \left(\left| u_{j+1}^n - u_j^n\right| + \left|u_j^n-u_{j-1}^n\right| \right)
	\end{align*}
	Thus, we have
	\begin{equation*}
		\Dx\sum_j |u_j^{n+1} - u_j^n | \leq C \TV(u^n)\Dt \leq C \TV(u_0)\Dt.
	\end{equation*}
	Finally, for $t\in[t^m,t^{m+1})$ and $s\in[t^l,t^{l+1})$ with $m>l$ we find
	\begin{align*}
		\int_\R \left|u_\Dx(x,t) - u_\Dx(x,s) \right| \diff x &= \Dx \sum_j |u_j^m - u_j^l|\\
		&\leq \sum_{n=l}^{m-1} \Dx\sum_j |u_j^{n+1}-u_j^n|\\
		&\leq C \TV(u_0) (m-l)\Dt\\
		&= C \TV(u_0) (t^m-t^l)\\
		&\leq C \TV(u_0) |t-s| + \mathcal{O}(\Dt).
	\end{align*}
	\end{enumerate}
\end{proof}

\subsection{Convergence of the numerical scheme}

With the help of the \textit{a priori} bounds derived in the previous section we can show that numerical solutions computed with the scheme~\eqref{num scheme} converge and that the limit is a weak solution. To that end, we will first show compactness of the scheme by applying Kolmogorov's theorem.
\begin{lemma}[Compactness]\label{lem: compactness}
	Let $u_0\in\Lone(\R)\cap\BV(\R)$ and let $u_\Dt(x,t) = u_j^n$ for $(x,t)\in \cell_j\times\timecell^n$, where $u_j^n$ is computed by the scheme \eqref{FE discretization}. Further, let $\lambda = \Dt/\Dx$ be fixed such that the CFL condition \eqref{CFL condition} is satisfied. Then there exists a sequence $(\Dt_k)_{k\in\N}$ and a function $u\in\mathcal{C}([0,T];\Lone(\R))$ such that $\Dt_k\to 0$ and $u_{\Dt_k}(t)$ converges to $u(t)$ in $\Lone(\R)$ uniformly for all $t\in[0,T]$.
\end{lemma}
\begin{proof}
	An application of Kolmogorov's compactness theorem \cite[Theorem A.11]{holden2015front} requires an $\Linfty$ bound, a TV bound, and $\Lone$ continuity in time as $\Dt\to 0$. In view of \Cref{lem: properties of the scheme}, we have
	\begin{align*}
		\|u_\Dt\|_{\Linfty(\R)} &\leq \|u_0\|_{\Linfty(\R)},\\
		\|u_\Dt (\cdot +\varepsilon,t) - u_\Dt(\cdot,t)\|_{\Lone(\R)} &\leq \varepsilon \TV(u_\Dt(\cdot,t)) \leq \varepsilon \TV(u_0), \\
		\left\| u_\Dt(\cdot,t) - u_\Dt(\cdot,s) \right\|_{\Lone(\R)} &\leq C \TV(u_0) |t-s| + \mathcal{O}(\Dt)
	\end{align*}
	as $\Dt\to0$. Kolmogorov's compactness theorem then ensures the existence of a subsequence $\Dt_k\to 0$ and a function $u\in \mathcal{C}([0,T];\Lone(\R))$ such that $u_{\Dt_k}(t)$ converges to $u(t)$ in $\Lone(\R)$ uniformly for all $t\in[0,T]$.
\end{proof}
We will now show that the limit is in fact a weak solution by proving the following Lax--Wendroff-type theorem (cf.~\cite{lax1960systems}).
\begin{theorem}[Convergence towards a weak solution]\label{thm: Lax--Wendroff-type theorem}
	Let $u_0\in\Lone(\R)\cap\BV(\R)$ and let $u_\Dt$ be computed by the scheme \eqref{FE discretization} with $\lambda=\Dt/\Dx$ fixed such that the discrete flux $g$ satisfies the CFL condition \eqref{CFL condition}. Then there exists a subsequence $\Dt_k\to 0$ such that $u_{\Dt_k}(t)$ converges in $\Loneloc(\R)$ for all $t$ to a function $u\in \mathcal{C}([0,T];\Lone(\R))$ which is a weak solution to \eqref{nonlocal model}.
	If we additionally assume $\int_0^\delta \frac{\omega_\delta(h)}{h}\diff h <\infty$ then the whole sequence $u_{\Dt}$ converges towards the unique entropy solution of~\eqref{nonlocal model}.
\end{theorem}
We want to emphasize here that the statement of \Cref{thm: Lax--Wendroff-type theorem} pertaining uniqueness of weak solutions is generally false for local conservation laws where weak solutions are not unique.

\begin{proof}
From \Cref{lem: compactness} we know that there exists a subsequence $\Dt_k\to 0$ and a function $u\in\mathcal{C}([0,T];\Lone(\R))$ such that $u_{\Dt_k}(t)$ converges uniformly in $\Lone(\R)$ to $u(t)$ for all $t\in[0,T]$. For ease of notation we will denote this subsequence by $\Dt$.
Assuming for the moment that $u$ is a weak solution of~\eqref{nonlocal model}, if $\int_0^\delta \frac{\omega_\delta(h)}{h}\diff h<\infty$ then, in view of \Cref{theorem 1}, the limit $u$ is in fact an entropy solution. Since entropy solutions of \eqref{nonlocal model} are unique (cf.~\cite[Thm. 2.3]{du2017nonlocal}), the whole sequence $u_\Dt$ converges.

We will now show that $u$ is indeed a weak solution.
To that end, we multiply \eqref{FE discretization}
by $-\Dx \phi_j^n$, where
\begin{equation*}
	\phi_j^n \coloneqq \frac{\phi(x_\jmhf,t^n)+\phi(x_\jphf,t^n)}{2}
\end{equation*}
for some test function $\phi\in\mathcal{C}_c^\infty(\R\times[0,T))$.
Summing over $n$ and $j$, we get
\begin{equation}
\begin{split}
	-\Dx \sum_{n=0}^N \sum_j (u_j^{n+1} - u_j^n) \phi_j^n = {}&\Dt \sum_{n=0}^N \sum_j  (g^n_\jphf - g^n_\jmhf)W_0 \phi_j^n\\
	&+\Dt \sum_{n=0}^N\sum_j\sum_{k=1}^{r+1}(g^n_{j,j+k} - g^n_{j-k,j}) \frac{1}{k}W_k \phi_j^n.
\end{split}
	\label{eqn: discrete weak formulation}
\end{equation}
Using summation by parts, the left-hand side of \eqref{eqn: discrete weak formulation} can be rewritten as
\begin{equation*}
	\Dx\sum_j u_j^0 \phi_j^0 + \Dx\Dt \sum_{n=1}^N \sum_j u_j^n \frac{\phi_j^n - \phi_j^{n-1}}{\Dt}
\end{equation*}
where we can pass to the limit $\Dt\to 0$ to get
\begin{equation*}
	\int_\R u_0(x)\phi(x,0)\diff x + \int_0^T \int_\R u \frac{\partial \phi}{\partial t}\diff x\diff t.
\end{equation*}
Thus it remains to show that the right-hand side of~\eqref{eqn: discrete weak formulation} converges to
\begin{equation*}
	\int_0^T\int_\R \int_0^\delta \frac{\phi-\tau_h\phi}{h} g(u,\tau_h u)\omega_\delta(h)\diff h\diff x\diff t
\end{equation*}
as $\Dx\to 0$.
Using the definition of the weights $W_k$, the right-hand side of~\eqref{eqn: discrete weak formulation} is equal to
\begin{align*}
	&\Dt\Dx \sum_{n=0}^N\sum_j \bigg[ \int_{\I_1}\left( \frac{g^n_\jphf-g^n_\jmhf}{\Dx}\Phi_0(h) + \frac{g^n_{j,j+1}-g^n_{j-1,j}}{\Dx} \Phi_1(h) \right)\omega_\delta(h)\diff h\\
	&\phantom{\Dt\Dx \sum_{n=0}^N\sum_j \bigg[} + \sum_{k=2}^{r+1} \int_{\I_k}\left( \frac{g^n_{j,j+(k-1)}-g^n_{j-(k-1),j}}{(k-1)\Dx} \Phi_{k-1}(h) + \frac{g^n_{j,j+k}-g^n_{j-k,j}}{k\Dx} \Phi_k(h) \right)\omega_\delta(h)\diff h \bigg]\phi_j^n\\
	&= \Dt\Dx\sum_{n=0}^N\sum_j \bigg[ \int_{\I_1}\left( g^n_\jphf \Phi_0(h) + g^n_{j,j+1}\Phi_1(h) \right) \frac{\phi_j^n - \phi_{j+1}^n}{\Dx}\omega_\delta(h)\diff h\\
	&\phantom{\mathrel{=}\Dt\Dx\sum_{n=0}^N\sum_j \bigg[}+\sum_{k=2}^{r+1} \int_{\I_k}\left( g^n_{j,j+(k-1)}\frac{\phi_j^n - \phi_{j+(k-1)}^n}{(k-1)\Dx} \Phi_{k-1}(h) + g^n_{j,j+k}\frac{\phi_j^n - \phi_{j+k}^n}{k\Dx} \Phi_k(h) \right)\omega_\delta(h)\diff h \bigg].
\end{align*}
We define
\begin{equation*}
	\widehat{G}(x,t,h) = \begin{cases}
		\left( g^n_\jphf \Phi_0(h) + g^n_{j,j+1}\Phi_1(h) \right) \frac{\phi_j^n - \phi_{j+1}^n}{\Dx}& h\in\I_1,\\
		g^n_{j,j+(k-1)}\frac{\phi_j^n - \phi_{j+(k-1)}^n}{(k-1)\Dx} \Phi_{k-1}(h) + g^n_{j,j+k}\frac{\phi_j^n - \phi_{j+k}^n}{k\Dx} \Phi_k(h)& h\in\I_k, k\in\{2,\ldots,r+1\}
	\end{cases}
\end{equation*}
for $x\in\cell_j$ and $t\in\timecell^n$.
Then the foregoing expression (and thus the right-hand side of~\eqref{eqn: discrete weak formulation}) is equal to
\begin{equation*}
	\sum_{n=0}^N\sum_j \sum_{k=1}^{r+1} \int_{\timecell^n}\int_{\cell_j} \int_{\I_k} \widehat{G}(x,t,h)\omega_\delta(h)\diff h\diff x\diff t = \int_0^T \int_\R\int_0^\delta \widehat{G}(x,t,h)\omega_\delta(h)\diff h\diff x\diff t.
\end{equation*}
We want to use the dominated convergence theorem in the variable $h$ to show that
\begin{align*}
	\lim_{\Dx\to 0} \int_0^T\int_\R\int_0^\delta \widehat{G}(x,t,h)\omega_\delta(h)\diff h\diff x\diff t &= \lim_{\Dx\to 0} \int_0^\delta \omega_\delta(h) \int_0^T\int_\R \widehat{G}(x,t,h)\diff x\diff t \diff h\\
	&= \int_0^\delta \omega_\delta(h) \lim_{\Dx\to 0} \int_0^T\int_\R \widehat{G}(x,t,h) \diff x\diff t\diff h\\
	&= \int_0^\delta \omega_\delta(h) \int_0^T\int_\R \frac{\phi-\tau_h \phi}{h} g(u,\tau_h u) \diff x\diff t \diff h\\
	&= \int_0^T\int_\R \int_0^\delta \frac{\phi-\tau_h\phi}{h} g(u,\tau_h u) \omega_\delta(h) \diff h\diff x\diff t.
\end{align*}
To that end, we need to verify that:
\begin{enumerate}[(i)]
	\item For a.e. $h\in(0,\delta)$
	\begin{equation*}
		\lim_{\Dx\to 0} \int_0^T\int_\R \widehat{G}(x,t,h) \diff x\diff t = \int_0^T\int_\R \frac{\phi-\tau_h \phi}{h} g(u,\tau_h u)\diff x\diff t.
	\end{equation*}
	\item There exists $\overline{G}\in\mathrm{L}^1(0,\delta)$ such that
	\begin{equation*}
		\left| \omega_\delta(h) \int_0^T\int_\R \widehat{G}(x,t,h)\diff x\diff t\right| \leq \overline{G}(h)
	\end{equation*}
	for all $h\in(0,\delta)$.
\end{enumerate}
\paragraph{Ad (ii):} Let the support of $\phi$ be in $(a,b)\times[0,T)$. Since
\begin{equation*}
	\left|\frac{\phi_j^n - \phi_{j+k}^n}{k\Dx}\right| \leq \|\phi_x\|_{\mathrm{L}^\infty(\R\times(0,T))}
\end{equation*}
and because of the Lipschitz continuity of $g$ and the $\mathrm{L}^\infty$ bound of $u_j^n$ we have
\begin{align*}
	\left| \omega_\delta(h) \int_0^T\int_\R \widehat{G}(x,t,h)\diff x\diff t\right| &\leq  \omega_\delta(h) \int_0^T\int_{a-\delta}^{b+\delta} \left|\widehat{G}(x,t,h)\right|\diff x\diff t\\
	&\leq C T (b-a+2\delta) \|\phi\|_{\mathrm{L}^\infty(\R\times(0,T))} \omega_\delta(h)
\end{align*}
which is integrable since $\omega_\delta\in\mathrm{L}^1(0,\delta)$.
\paragraph{Ad (i):} By using exactly the same steps as in~\cite[Lem. 4.9]{du2017nonlocal}\footnote{Note that the definition of $\phi_j^n$ used here ($\phi_j^n = \frac{1}{2}(\phi(x_\jmhf,t^n)+\phi(x_\jphf,t^n)\geq 0$) is not the same as \cite[Eq. (4.15)]{du2017nonlocal} ($\phi_j^n = \frac{1}{\Dx}(\phi(x_\jmhf,t^n)-\phi(x_\jphf,t^n)) \not\geq 0$). After careful inspection of the proofs of \cite{du2017nonlocal} we consider this a misprint.} except for substituting the entropy flux $q$ by $g$ we see that for a.e. $h\in(0,\delta)$
\begin{equation*}
	\lim_{\Dx\to 0} \Dx\Dt\sum_{n=0}^N\sum_j \frac{\phi_j^n - \phi_{j+k}^n}{k\Dx} g^n_{j,j+k} = \int_0^T\int_\R \frac{\phi-\tau_h \phi}{h} g(u,\tau_h u)\diff x\diff t
\end{equation*}
where, for each $\Dx$, $k$ is such that $(k-1)\Dx< h <k\Dx$. Thus it suffices to show that
\begin{equation*}
	\lim_{\Dx\to 0} \left| \int_0^T\int_\R \widehat{G}(x,t,h)\diff x\diff t - \Dx\Dt\sum_{n=0}^N\sum_j \frac{\phi_j^n - \phi_{j+k}^n}{k\Dx} g^n_{j,j+k} \right| = 0
\end{equation*}
for a.e. $h\in(0,\delta)$. Now, we fix $h\in(0,\delta)$ and without restriction assume that $\Dx<h$. Then there exists $k\in\{2,\ldots,r+1\}$ such that $h\in[(k-1)\Dx,k\Dx)$. We have
\begin{align*}
	&\left| \int_0^T\int_\R \widehat{G}(x,t,h)\diff x\diff t - \Dx\Dt\sum_{n=0}^N\sum_j \frac{\phi_j^n - \phi_{j+k}^n}{k\Dx} g^n_{j,j+k} \right|\\
	&\leq \Dx\Dt \sum_{n=0}^N\sum_j \Bigg| \frac{\phi_j^n - \phi_{j+(k-1)}^n}{(k-1)\Dx} g^n_{j,j+(k-1)} \Phi_{k-1}(h) + \frac{\phi_j^n-\phi_{j+k}^n}{k\Dx} g^n_{j,j+k} {\underbrace{(\Phi_k(h) - 1)}_{=-\Phi_{k-1}(h)}} \Bigg|\\
	&\leq \Dx\Dt\sum_{n=0}^N\sum_j \bigg[ \left| \frac{\phi_j^n-\phi_{j+(k-1)}^n}{(k-1)\Dx} - \frac{\phi_j^n - \phi_{j+k}^n}{k\Dx} \right| |g^n_{j,j+(k-1)}|
	+\left|\frac{\phi_j^n-\phi_{j+k}^n}{k\Dx}\right| \left| g^n_{j,j+(k-1)}-g^n_{j,j+k} \right| \bigg] \Phi_{k-1}(h).
\end{align*}
Repeated application of the mean value theorem shows that
\begin{equation*}
	\left|\frac{\phi_j^n-\phi_{j+(k-1)}^n}{(k-1)\Dx} - \frac{\phi_j^n - \phi_{j+k}^n}{k\Dx}\right| \leq 2\|\phi_x\|_{\mathrm{L}^\infty(\R\times(0,T))} \frac{1}{k-1}.
\end{equation*}
Since $h\leq k\Dx$ we have $k-1 \geq \frac{h}{\Dx}-1$ and thus we can bound $\frac{1}{k-1}$ from above by $\frac{\Dx}{h-\Dx}$ which converges to zero as $\Dx\to 0$.
Using the fact that $|g^n_{j,j+(k-1)}|$ is bounded, that $\Phi_{k-1}(h)\leq 1$, and that the sum over $j$ is finite since $\phi$ has compact support, we get
\begin{equation*}
	\Dx\Dt\sum_{n=0}^N\sum_j \left| \frac{\phi_j^n-\phi_{j+(k-1)}^n}{(k-1)\Dx} - \frac{\phi_j^n - \phi_{j+k}^n}{k\Dx} \right| |g^n_{j,j+(k-1)}| \Phi_{k-1}(h) \leq C T \frac{\Dx}{h-\Dx} \to 0
\end{equation*}
as $\Dx\to 0$.
On the other hand, since
\begin{equation*}
	\left| \frac{\phi_j^n-\phi_{j+k}^n}{k\Dx}\right| \leq \|\phi_x\|_{\mathrm{L}^\infty(\R\times(0,T))}
\end{equation*}
we can use the Lipschitz continuity of $g$ and the total variation bound of $u_j^n$ to get
\begin{align*}
	\Dx\Dt\sum_{n=0}^N\sum_j \left| \frac{\phi_j^n-\phi_{j+k}^n}{k\Dx}\right| & |g^n_{j,j+(k-1)} - g^n_{j,j+k}| \Phi_{k-1}(h)\\
	&\leq C\|\phi_x\|_{\mathrm{L}^\infty(\R\times(0,T))} \Dx\Dt\sum_{n=0}^N\sum_j |u_{j+k}-u_{j+k-1}|\\
	&\leq C T\|\phi_x\|_{\mathrm{L}^\infty(\R\times(0,T))} \TV(u_0) \Dx \to 0
\end{align*}
as $\Dx\to 0$.

Hence, we can apply the dominated convergence theorem to obtain
\begin{equation*}
	\lim_{\Dx\to 0} \int_0^T\int_\R\int_0^\delta \widehat{G}(x,t,h)\omega_\delta(h)\diff h\diff x\diff t
	= \int_0^T\int_\R \int_0^\delta \frac{\phi-\tau_h\phi}{h} g(u,\tau_h u) \omega_\delta(h) \diff h\diff x\diff t.
\end{equation*}
showing that $u$ is a weak solution.
\end{proof}

\section{Numerical experiments}\label{sec: Numerical experiments}
%
In this section we present a series of numerical experiments for the second-order scheme~\eqref{num scheme}. Our aims are the following:
\begin{itemize}
	\begin{samepage}
	\item to compare the second-order scheme developed in this paper to the first-order scheme~\eqref{first-order numerical scheme} presented in \cite{du2017nonlocal,du2017numerical};
	\item to compare convergence rates of the second-order scheme for the nonlocal model and the local conservation law;
	\item to illustrate our regularity results from \Cref{sec: Regularity} numerically by comparing numerical solutions of the nonlocal model and the local conservation law in shock and stationary shock regimes;
	\item to numerically investigate how shock formation in the nonlocal model depends on the choice of numerical flux $g$;
	\item to numerically verify that the second-order scheme for the nonlocal model is asymptotically compatible with the entropy solution of the local conservation law.
	\end{samepage}
\end{itemize}
In all our numerical experiments we employ Burgers' flux
\begin{equation*}
	f(u) = \frac{u^2}{2},
\end{equation*}
and -- except when stated otherwise -- use the Godunov flux
\begin{equation*}
	g(u,v) = \max(f(\max(u,0)), f(\min(v,0)))
\end{equation*}
as the numerical flux function \cite{godunov1959difference}.
The slopes $\sigma_j$ which are used in the second-order scheme are computed with the $\operatorname{minmod}$ limiter
\begin{equation*}
	\sigma_j^n = \minmod( u_{j+1}^n - u_j^n, u_j^n - u_{j-1}^n),
\end{equation*}
where
\begin{equation*}
	\minmod(a,b) = \begin{cases}
		a &\text{if }|a|\leq|b| \text{ and }ab>0,\\
		b &\text{if }|a|>|b| \text{ and }ab>0,\\
		0 &\text{if }ab\leq0,
	\end{cases}
\end{equation*}
and we will use the nonlocal interaction kernel
\begin{equation*}
	\omega_\delta (h) = \frac{1+p}{\delta^{1+p}} h^p \ind_{(0,\delta)}(h)
\end{equation*}
for various powers $p>-1$. In all numerical experiments we will use periodic boundary conditions (except when considering the Riemann problem where we use outflow boundary conditions).

\subsection{Experiment 1: Comparison to the first-order scheme}
First, we compare the second-order scheme \eqref{num scheme} to the first-order scheme presented in \cite{du2017nonlocal,du2017numerical}.
To that end, we will consider the initial datum
\begin{equation}
	u_0^1(x) = \frac{1+\sin(2\pi x)}{2}
	\label{numerical experiments: u01}
\end{equation}
on the unit interval, nonlocal horizon $\delta = 0.125$, end time $T = 0.3$, and CFL parameter $\lambda = 0.8$.
\Cref{fig: first- and second-order numerical approximations} shows numerical approximations computed by the first-order scheme presented in \cite{du2017nonlocal,du2017numerical} (blue squares) and the second-order scheme presented in this paper (red circles). Here and in the subsequent experiments the gray dashed line indicates the initial datum.
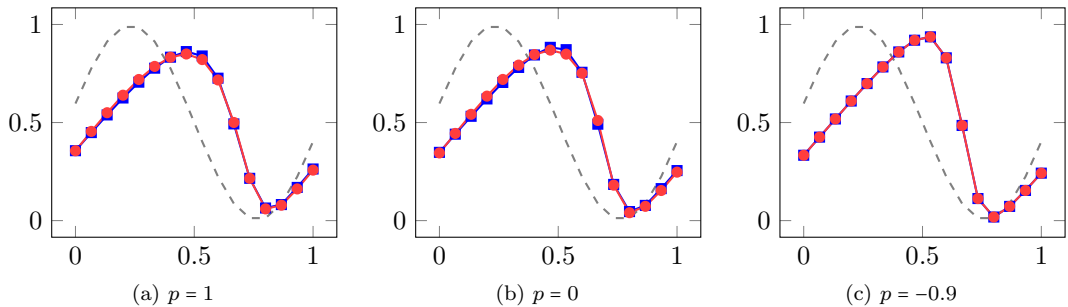
\begin{figure}[t]
\centering
\subfloat[$p=1$]{
\begin{tikzpicture}
\begin{axis}[mark size=1.7pt]
\addplot+[gray, thick, dashed, mark=none] coordinates{
(0.000000000000000000e+00 , 5.969195893703570510e-01)
(6.666666666666666574e-02 , 7.760036392076995471e-01)
(1.333333333333333315e-01 , 9.130686369548897119e-01)
(2.000000000000000111e-01 , 9.872476792022164549e-01)
(2.666666666666666630e-01 , 9.872476792022162329e-01)
(3.333333333333333148e-01 , 9.130686369548897119e-01)
(4.000000000000000222e-01 , 7.760036392076994360e-01)
(4.666666666666666741e-01 , 5.969195893703576061e-01)
(5.333333333333333259e-01 , 4.030804106296423939e-01)
(5.999999999999999778e-01 , 2.239963607923014521e-01)
(6.666666666666666297e-01 , 8.693136304511028811e-02)
(7.333333333333332815e-01 , 1.275232079778287897e-02)
(8.000000000000000444e-01 , 1.275232079778465533e-02)
(8.666666666666666963e-01 , 8.693136304511028811e-02)
(9.333333333333333481e-01 , 2.239963607922996758e-01)
(1.000000000000000000e+00 , 4.030804106296423939e-01)
};
\addplot[mark=square*, blue, thick] coordinates{
(0.000000000000000000e+00 , 3.564778635398128670e-01)
(6.666666666666666574e-02 , 4.488022210724889050e-01)
(1.333333333333333315e-01 , 5.388352100474813611e-01)
(2.000000000000000111e-01 , 6.251966765110856805e-01)
(2.666666666666666630e-01 , 7.057966473148437059e-01)
(3.333333333333333148e-01 , 7.771913828621549447e-01)
(4.000000000000000222e-01 , 8.331910636809136683e-01)
(4.666666666666666741e-01 , 8.618081808634475705e-01)
(5.333333333333333259e-01 , 8.394120780507612345e-01)
(5.999999999999999778e-01 , 7.260954066792674189e-01)
(6.666666666666666297e-01 , 4.930592712559870860e-01)
(7.333333333333332815e-01 , 2.155948460438227299e-01)
(8.000000000000000444e-01 , 6.426270642518244758e-02)
(8.666666666666666963e-01 , 8.187251884270849156e-02)
(9.333333333333333481e-01 , 1.692131848845644104e-01)
(1.000000000000000000e+00 , 2.631907419254773117e-01)
};
\addplot[mark=*, myred, thick]
coordinates{%
(0.000000000000000000e+00 , 3.562561111681398218e-01)
(6.666666666666666574e-02 , 4.548421183845950955e-01)
(1.333333333333333315e-01 , 5.501127662538785046e-01)
(2.000000000000000111e-01 , 6.394206715249751705e-01)
(2.666666666666666630e-01 , 7.195438647187415437e-01)
(3.333333333333333148e-01 , 7.861835470923774771e-01)
(4.000000000000000222e-01 , 8.329366446293657855e-01)
(4.666666666666666741e-01 , 8.495283710478704897e-01)
(5.333333333333333259e-01 , 8.210408164993797664e-01)
(5.999999999999999778e-01 , 7.169771807034363009e-01)
(6.666666666666666297e-01 , 5.003949381318744116e-01)
(7.333333333333332815e-01 , 2.155283562051500679e-01)
(8.000000000000000444e-01 , 5.974168092125764495e-02)
(8.666666666666666963e-01 , 7.855052732049241926e-02)
(9.333333333333333481e-01 , 1.615422071936192250e-01)
(1.000000000000000000e+00 , 2.574001982048459425e-01)
};
\end{axis}
\end{tikzpicture}
\label{fig: p=1}
}%
\subfloat[$p=0$]{
\begin{tikzpicture}
\begin{axis}[mark size=1.7pt]
\addplot+[gray, thick, dashed, mark=none] coordinates{
(0.000000000000000000e+00 , 5.969195893703570510e-01)
(6.666666666666666574e-02 , 7.760036392076995471e-01)
(1.333333333333333315e-01 , 9.130686369548897119e-01)
(2.000000000000000111e-01 , 9.872476792022164549e-01)
(2.666666666666666630e-01 , 9.872476792022162329e-01)
(3.333333333333333148e-01 , 9.130686369548897119e-01)
(4.000000000000000222e-01 , 7.760036392076994360e-01)
(4.666666666666666741e-01 , 5.969195893703576061e-01)
(5.333333333333333259e-01 , 4.030804106296423939e-01)
(5.999999999999999778e-01 , 2.239963607923014521e-01)
(6.666666666666666297e-01 , 8.693136304511028811e-02)
(7.333333333333332815e-01 , 1.275232079778287897e-02)
(8.000000000000000444e-01 , 1.275232079778465533e-02)
(8.666666666666666963e-01 , 8.693136304511028811e-02)
(9.333333333333333481e-01 , 2.239963607922996758e-01)
(1.000000000000000000e+00 , 4.030804106296423939e-01)
};
\addplot[mark=square*, blue, thick] coordinates{
(0.000000000000000000e+00 , 3.479653518087731379e-01)
(6.666666666666666574e-02 , 4.406006176839779576e-01)
(1.333333333333333315e-01 , 5.318410640098089059e-01)
(2.000000000000000111e-01 , 6.203927167803406917e-01)
(2.666666666666666630e-01 , 7.044318024202411710e-01)
(3.333333333333333148e-01 , 7.809811784000021406e-01)
(4.000000000000000222e-01 , 8.445036989438448050e-01)
(4.666666666666666741e-01 , 8.834043770000442608e-01)
(5.333333333333333259e-01 , 8.715533458327894278e-01)
(5.999999999999999778e-01 , 7.562901311414955696e-01)
(6.666666666666666297e-01 , 4.916235468611558357e-01)
(7.333333333333332815e-01 , 1.845823483074860438e-01)
(8.000000000000000444e-01 , 4.607447738129345771e-02)
(8.666666666666666963e-01 , 7.748392933835013263e-02)
(9.333333333333333481e-01 , 1.631544884233862569e-01)
(1.000000000000000000e+00 , 2.551169256670102747e-01)
};
\addplot[mark=*, myred, thick] coordinates{
(0.000000000000000000e+00 , 3.450770751617929788e-01)
(6.666666666666666574e-02 , 4.445018217053804399e-01)
(1.333333333333333315e-01 , 5.419289210653781552e-01)
(2.000000000000000111e-01 , 6.346745368159083522e-01)
(2.666666666666666630e-01 , 7.197573747633537744e-01)
(3.333333333333333148e-01 , 7.927198854492334412e-01)
(4.000000000000000222e-01 , 8.465382511647840946e-01)
(4.666666666666666741e-01 , 8.693311674670067823e-01)
(5.333333333333333259e-01 , 8.487954986471430985e-01)
(5.999999999999999778e-01 , 7.514991881481456470e-01)
(6.666666666666666297e-01 , 5.106047432035554223e-01)
(7.333333333333332815e-01 , 1.807902077664701457e-01)
(8.000000000000000444e-01 , 4.034807502045705679e-02)
(8.666666666666666963e-01 , 7.277634046266115431e-02)
(9.333333333333333481e-01 , 1.535549682148721340e-01)
(1.000000000000000000e+00 , 2.471019449438567606e-01)
};
\end{axis}
\end{tikzpicture}
\label{fig: p=0}
}%
\subfloat[$p=-0.9$]{
\begin{tikzpicture}
\begin{axis}[mark size=1.7pt]
\addplot+[gray, thick, dashed, mark=none] coordinates{
(0.000000000000000000e+00 , 5.969195893703570510e-01)
(6.666666666666666574e-02 , 7.760036392076995471e-01)
(1.333333333333333315e-01 , 9.130686369548897119e-01)
(2.000000000000000111e-01 , 9.872476792022164549e-01)
(2.666666666666666630e-01 , 9.872476792022162329e-01)
(3.333333333333333148e-01 , 9.130686369548897119e-01)
(4.000000000000000222e-01 , 7.760036392076994360e-01)
(4.666666666666666741e-01 , 5.969195893703576061e-01)
(5.333333333333333259e-01 , 4.030804106296423939e-01)
(5.999999999999999778e-01 , 2.239963607923014521e-01)
(6.666666666666666297e-01 , 8.693136304511028811e-02)
(7.333333333333332815e-01 , 1.275232079778287897e-02)
(8.000000000000000444e-01 , 1.275232079778465533e-02)
(8.666666666666666963e-01 , 8.693136304511028811e-02)
(9.333333333333333481e-01 , 2.239963607922996758e-01)
(1.000000000000000000e+00 , 4.030804106296423939e-01)
};
\addplot[mark=square*, blue, thick] coordinates{
(0.000000000000000000e+00 , 3.332342831231698277e-01)
(6.666666666666666574e-02 , 4.257089822601813900e-01)
(1.333333333333333315e-01 , 5.182436554781771898e-01)
(2.000000000000000111e-01 , 6.097168486908681739e-01)
(2.666666666666666630e-01 , 6.987873694819294323e-01)
(3.333333333333333148e-01 , 7.834440187650064669e-01)
(4.000000000000000222e-01 , 8.599675221421915206e-01)
(4.666666666666666741e-01 , 9.197205255694425041e-01)
(5.333333333333333259e-01 , 9.369595085029160941e-01)
(5.999999999999999778e-01 , 8.299564164139434652e-01)
(6.666666666666666297e-01 , 4.851023523371439383e-01)
(7.333333333333332815e-01 , 1.125207246120232540e-01)
(8.000000000000000444e-01 , 1.833438169074281046e-02)
(8.666666666666666963e-01 , 7.241669266369529789e-02)
(9.333333333333333481e-01 , 1.538535765697008617e-01)
(1.000000000000000000e+00 , 2.420331416988674711e-01)
};
\addplot[mark=*, myred, thick] coordinates{
(0.000000000000000000e+00 , 3.332342831231698277e-01)
(6.666666666666666574e-02 , 4.257089822601813900e-01)
(1.333333333333333315e-01 , 5.182436554781771898e-01)
(2.000000000000000111e-01 , 6.097168486908681739e-01)
(2.666666666666666630e-01 , 6.987873694819294323e-01)
(3.333333333333333148e-01 , 7.834440187650064669e-01)
(4.000000000000000222e-01 , 8.599675221421915206e-01)
(4.666666666666666741e-01 , 9.197205255694425041e-01)
(5.333333333333333259e-01 , 9.369595085029160941e-01)
(5.999999999999999778e-01 , 8.299564164139434652e-01)
(6.666666666666666297e-01 , 4.851023523371439383e-01)
(7.333333333333332815e-01 , 1.125207246120232540e-01)
(8.000000000000000444e-01 , 1.833438169074281046e-02)
(8.666666666666666963e-01 , 7.241669266369529789e-02)
(9.333333333333333481e-01 , 1.538535765697008617e-01)
(1.000000000000000000e+00 , 2.420331416988674711e-01)
};
\end{axis}
\end{tikzpicture}
\label{fig: p=-0.9}
}%
\caption{Experiment 1. First- and second-order numerical approximations for $\Dx=1/16$ (blue squares and red circles, respectively) and various values of $p$.}
\label{fig: first- and second-order numerical approximations}
\end{figure}
\Cref{tbl: Errors and convergence rates} shows the $\Lone$ error against an approximation on a very fine grid ($n=1024$ spatial cells) as well as the observed order of convergence of the second-order method.
\begin{table}[t]
\centering
\subfloat[$p=1$]{
\begin{tabular}{@{}rcc@{}}
  \toprule
  \multicolumn{1}{c}{$n$} & $\Lone$ error & $\Lone$ OOC\\
  \midrule
	 $  8$&  $\num{1.440e-02}$ & -- \\
	 $ 16$&  $\num{1.948e-03}$ & $2.89$ \\
	 $ 32$&  $\num{4.092e-04}$ & $2.25$ \\
	 $ 64$&  $\num{9.264e-05}$ & $2.14$ \\
	 $128$&  $\num{2.201e-05}$ & $2.07$ \\
	 $256$&  $\num{5.146e-06}$ & $2.10$ \\
	 $512$&  $\num{1.021e-06}$ & $2.33$ \\
  \bottomrule
\end{tabular}
\label{tbl: p=1}
}%
\subfloat[$p=0$]{
\begin{tabular}{@{}rcc@{}}
  \toprule
  \multicolumn{1}{c}{$n$} & $\Lone$ error & $\Lone$ OOC\\
  \midrule
	$  8$ &  $\num{2.212e-02}$ & -- \\
	$ 16$ &  $\num{3.686e-03}$ & $2.59$ \\
	$ 32$ &  $\num{7.048e-04}$ & $2.39$ \\
	$ 64$ &  $\num{1.473e-04}$ & $2.26$ \\
	$128$ &  $\num{3.277e-05}$ & $2.17$ \\
	$256$ &  $\num{7.348e-06}$ & $2.16$ \\
	$512$ &  $\num{1.426e-06}$ & $2.37$ \\
  \bottomrule
\end{tabular}
\label{tbl: p=0}
}%
\subfloat[$p=-0.9$]{
\begin{tabular}{@{}rcc@{}}
  \toprule
  \multicolumn{1}{c}{$n$} & $\Lone$ error & $\Lone$ OOC\\
  \midrule
	$  8$ &  $\num{5.250e-02}$ & -- \\
	$ 16$ &  $\num{1.951e-02}$ & $1.43$ \\
	$ 32$ &  $\num{6.303e-03}$ & $1.63$ \\
	$ 64$ &  $\num{1.695e-03}$ & $1.89$ \\
	$128$ &  $\num{4.284e-04}$ & $1.98$ \\
	$256$ &  $\num{1.003e-04}$ & $2.09$ \\
	$512$ &  $\num{1.982e-05}$ & $2.34$ \\
  \bottomrule
\end{tabular}
\label{tbl: p=-0.9}
}%
\caption{Experiment 1. $\Lone$ errors and observed order of convergence for various values of $p$.}
\label{tbl: Errors and convergence rates}
\end{table}

\subsection{Experiment 2: Comparison to the local conservation law -- convergence rates}
Since for the local conservation law \eqref{local claw} the initial datum \eqref{numerical experiments: u01} will lead to a discontinuity at time
\begin{equation*}
	t^* = -\frac{1}{\min u_0'(x)} = \frac{1}{\pi}\approx 0.318,
\end{equation*}
it is of interest to investigate the observed order of convergence for times $t>t^*$. To that end, we conduct a second numerical experiment using the same parameters as in Experiment 1, except setting $T = 0.5$.
\begin{figure}[t]
\centering
\subfloat[Local conservation law]{
\begin{tikzpicture}
\begin{axis}[mark size=1.7pt]
\addplot+[gray, thick, dashed, mark=none] coordinates{
(0.000000000000000000e+00 , 5.969195893703570510e-01)
(6.666666666666666574e-02 , 7.760036392076995471e-01)
(1.333333333333333315e-01 , 9.130686369548897119e-01)
(2.000000000000000111e-01 , 9.872476792022164549e-01)
(2.666666666666666630e-01 , 9.872476792022162329e-01)
(3.333333333333333148e-01 , 9.130686369548897119e-01)
(4.000000000000000222e-01 , 7.760036392076994360e-01)
(4.666666666666666741e-01 , 5.969195893703576061e-01)
(5.333333333333333259e-01 , 4.030804106296423939e-01)
(5.999999999999999778e-01 , 2.239963607923014521e-01)
(6.666666666666666297e-01 , 8.693136304511028811e-02)
(7.333333333333332815e-01 , 1.275232079778287897e-02)
(8.000000000000000444e-01 , 1.275232079778465533e-02)
(8.666666666666666963e-01 , 8.693136304511028811e-02)
(9.333333333333333481e-01 , 2.239963607922996758e-01)
(1.000000000000000000e+00 , 4.030804106296423939e-01)
};
\addplot[mark=*, myred, thick]
coordinates{%
(0.000000000000000000e+00 , 2.204888385980622867e-01 )
(3.225806451612903136e-02 , 2.561493602615871890e-01 )
(6.451612903225806273e-02 , 2.925192397389122512e-01 )
(9.677419354838709409e-02 , 3.294512137995458945e-01 )
(1.290322580645161255e-01 , 3.668264986105941317e-01 )
(1.612903225806451568e-01 , 4.045458166252191012e-01 )
(1.935483870967741882e-01 , 4.425192135065961208e-01 )
(2.258064516129032195e-01 , 4.806633090713120504e-01 )
(2.580645161290322509e-01 , 5.189136136098003727e-01 )
(2.903225806451612545e-01 , 5.572516965040674730e-01 )
(3.225806451612903136e-01 , 5.957447501856285577e-01 )
(3.548387096774193727e-01 , 6.344341824839205835e-01 )
(3.870967741935483764e-01 , 6.722208375268129421e-01 )
(4.193548387096773800e-01 , 7.085031243953056057e-01 )
(4.516129032258064391e-01 , 7.437202535156368643e-01 )
(4.838709677419354982e-01 , 7.785655683281522688e-01 )
(5.161290322580645018e-01 , 8.134492431414165337e-01 )
(5.483870967741935054e-01 , 8.486908811232074523e-01 )
(5.806451612903225090e-01 , 8.852093894738725099e-01 )
(6.129032258064516236e-01 , 9.233327334100670525e-01 )
(6.451612903225806273e-01 , 9.476528015176124420e-01 )
(6.774193548387096309e-01 , 9.515241460564158871e-01 )
(7.096774193548387455e-01 , 9.343732403199843661e-01 )
(7.419354838709677491e-01 , 7.935735802691675778e-01 )
(7.741935483870967527e-01 , 2.359572033771295263e-01 )
(8.064516129032257563e-01 , 1.874031345754028494e-02 )
(8.387096774193547599e-01 , 3.648731261373640983e-02 )
(8.709677419354838745e-01 , 6.140226596995664032e-02 )
(9.032258064516128782e-01 , 8.944671655370965702e-02 )
(9.354838709677418818e-01 , 1.198482834655417717e-01 )
(9.677419354838709964e-01 , 1.520739246553469148e-01 )
(1.000000000000000000e+00 , 1.857204478341402121e-01 )
};
\end{axis}
\end{tikzpicture}
}%
\subfloat[Nonlocal model]{
\begin{tikzpicture}
\begin{axis}[mark size=1.7pt]
\addplot+[gray, thick, dashed, mark=none] coordinates{
(0.000000000000000000e+00 , 5.969195893703570510e-01)
(6.666666666666666574e-02 , 7.760036392076995471e-01)
(1.333333333333333315e-01 , 9.130686369548897119e-01)
(2.000000000000000111e-01 , 9.872476792022164549e-01)
(2.666666666666666630e-01 , 9.872476792022162329e-01)
(3.333333333333333148e-01 , 9.130686369548897119e-01)
(4.000000000000000222e-01 , 7.760036392076994360e-01)
(4.666666666666666741e-01 , 5.969195893703576061e-01)
(5.333333333333333259e-01 , 4.030804106296423939e-01)
(5.999999999999999778e-01 , 2.239963607923014521e-01)
(6.666666666666666297e-01 , 8.693136304511028811e-02)
(7.333333333333332815e-01 , 1.275232079778287897e-02)
(8.000000000000000444e-01 , 1.275232079778465533e-02)
(8.666666666666666963e-01 , 8.693136304511028811e-02)
(9.333333333333333481e-01 , 2.239963607922996758e-01)
(1.000000000000000000e+00 , 4.030804106296423939e-01)
};
\addplot[mark=*, myred, thick] coordinates{
(0.000000000000000000e+00 , 2.430624738089741110e-01 )
(3.225806451612903136e-02 , 2.798558576004241405e-01 )
(6.451612903225806273e-02 , 3.170309965090217785e-01 )
(9.677419354838709409e-02 , 3.544660308005898597e-01 )
(1.290322580645161255e-01 , 3.920526511614205734e-01 )
(1.612903225806451568e-01 , 4.296728724978709502e-01 )
(1.935483870967741882e-01 , 4.671797423773145219e-01 )
(2.258064516129032195e-01 , 5.044723753991082749e-01 )
(2.580645161290322509e-01 , 5.414426275060750804e-01 )
(2.903225806451612545e-01 , 5.779776397019382728e-01 )
(3.225806451612903136e-01 , 6.139484264330227203e-01 )
(3.548387096774193727e-01 , 6.492066659541855245e-01 )
(3.870967741935483764e-01 , 6.835727181948154652e-01 )
(4.193548387096773800e-01 , 7.168166644094355888e-01 )
(4.516129032258064391e-01 , 7.486460691150732538e-01 )
(4.838709677419354982e-01 , 7.786716902765087323e-01 )
(5.161290322580645018e-01 , 8.063329478046330490e-01 )
(5.483870967741935054e-01 , 8.307674291459218319e-01 )
(5.806451612903225090e-01 , 8.504162981855258874e-01 )
(6.129032258064516236e-01 , 8.621692352508099066e-01 )
(6.451612903225806273e-01 , 8.626790383351845559e-01 )
(6.774193548387096309e-01 , 8.450770593414986820e-01 )
(7.096774193548387455e-01 , 7.808537051774704363e-01 )
(7.419354838709677491e-01 , 6.091009959416920916e-01 )
(7.741935483870967527e-01 , 3.252010343346709575e-01 )
(8.064516129032257563e-01 , 1.425531712704898590e-01 )
(8.387096774193547599e-01 , 8.410206732756221604e-02 )
(8.709677419354838745e-01 , 8.251627362716595915e-02 )
(9.032258064516128782e-01 , 1.051203183823645748e-01 )
(9.354838709677418818e-01 , 1.369209008351686685e-01 )
(9.677419354838709964e-01 , 1.713058447994429412e-01 )
(1.000000000000000000e+00 , 2.068081784946200763e-01 )
};
\end{axis}
\end{tikzpicture}
}%
\caption{Experiment 2. Second-order numerical approximations of the local conservation law and the nonlocal model ($\delta = 0.125$, $p=-0.5$) for $\Dx=1/32$ and $\lambda =0.8$.}
\label{fig: comparison local and nonlocal shock}
\end{figure}
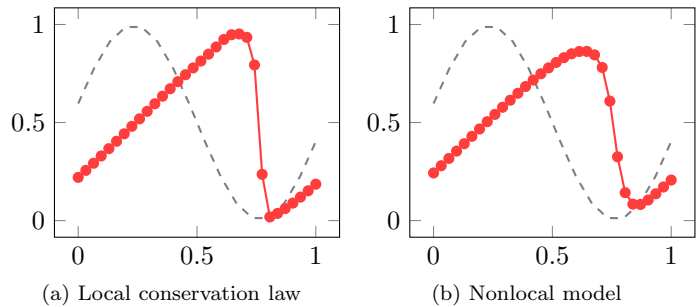
\begin{table}[t]
\centering
\subfloat[Local conservation law]{
\begin{tabular}{@{}rcc@{}}
  \toprule
  \multicolumn{1}{c}{$n$} & $\Lone$ error & $\Lone$ OOC\\
  \midrule
	 $  8$&  $\num{7.553e-02}$ & -- \\
	 $ 16$&  $\num{3.484e-02}$ & $1.12$ \\
	 $ 32$&  $\num{1.645e-02}$ & $1.08$ \\
	 $ 64$&  $\num{7.651e-03}$ & $1.10$ \\
	 $128$&  $\num{3.416e-03}$ & $1.16$ \\
	 $256$&  $\num{1.415e-03}$ & $1.27$ \\
	 $512$&  $\num{4.638e-04}$ & $1.61$ \\
  \bottomrule
\end{tabular}
\label{tbl: local}
}%
\subfloat[Nonlocal model]{
\centering
\begin{tabular}{@{}rcc@{}}
  \toprule
  \multicolumn{1}{c}{$n$} & $\Lone$ error & $\Lone$ OOC\\
  \midrule
	$  8$ &  $\num{3.904e-02}$ & -- \\
	$ 16$ &  $\num{9.936e-03}$ & $1.97$ \\
	$ 32$ &  $\num{2.784e-03}$ & $1.84$ \\
	$ 64$ &  $\num{6.115e-04}$ & $2.19$ \\
	$128$ &  $\num{1.208e-04}$ & $2.34$ \\
	$256$ &  $\num{2.295e-05}$ & $2.40$ \\
	$512$ &  $\num{3.772e-06}$ & $2.61$ \\
  \bottomrule
\end{tabular}
\label{tbl: nonlocal}
}%
\caption{Experiment 2. $\Lone$ errors and observed order of convergence for the local conservation law and the nonlocal model ($\delta = 0.125$, $p=-0.5$).}
\label{tbl: Errors and convergence rates local/nonlocal comparison}
\end{table}
\Cref{fig: comparison local and nonlocal shock} shows the numerical approximations computed with the second-order scheme \eqref{num scheme} for the local ($\delta=0$) conservation law (left) and the nonlocal model (right), with $\delta = 0.125$ and $p=-0.5$ in the latter case, while \Cref{tbl: Errors and convergence rates local/nonlocal comparison} shows the $\Lone$ error against an approximation on a very fine grid ($n=1024$) and the observed order of convergence.

The solution of the local conservation law has a shock moving to the right which will cause the second-order scheme to lose its second-order convergence rate. On the other hand, \Cref{thm: stationarity of discontinuities of weak solutions} shows that the nonlocal model cannot exhibit non-stationary shocks, so the second-order method will retain its second-order convergence rate in this case. We can clearly see this when comparing the experimental convergence rates of \Cref{tbl: local} to those of \Cref{tbl: nonlocal}.



\subsection{Experiment 3: Comparison to the local conservation law -- shocks}
We want to illustrate the theoretical findings of \Cref{sec: Regularity} further with a third experiment.
To that end, we consider the same parameters as before, except using the initial datum
\begin{equation*}
	u_0^2(x) = -\sin(\pi x)
\end{equation*}
on the interval $[-1,1]$. \Cref{fig: stationary shock} shows numerical solutions computed with the second-order scheme~\eqref{num scheme} for the local conservation law (blue) and for the nonlocal model (red) at $T=0.5$.  We observe a stationary shock at the origin both in the local and the nonlocal model.
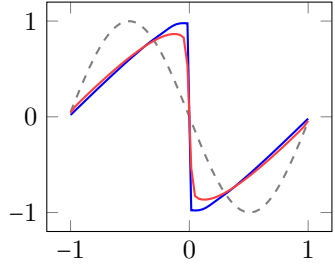
\begin{figure}[t]
\centering
\begin{tikzpicture}
\begin{axis}[mark size=1.7pt]
\addplot+[gray, thick, dashed, mark=none] coordinates{
(-1.000000000000000000e+00, 4.904797135733937608e-02)
(-9.682539682539682557e-01, 1.466715552954981905e-01)
(-9.365079365079365115e-01, 2.428826119773334347e-01)
(-9.047619047619047672e-01, 3.367545763326854069e-01)
(-8.730158730158730229e-01, 4.273834100291765692e-01)
(-8.412698412698412787e-01, 5.138963078555496278e-01)
(-8.095238095238095344e-01, 5.954601033129761634e-01)
(-7.777777777777777901e-01, 6.712892924638949044e-01)
(-7.460317460317460458e-01, 7.406535987643311225e-01)
(-7.142857142857143016e-01, 8.028850060262248123e-01)
(-6.825396825396825573e-01, 8.573841917784963407e-01)
(-6.507936507936508130e-01, 9.036262990700658637e-01)
(-6.190476190476190688e-01, 9.411659911292082903e-01)
(-5.873015873015873245e-01, 9.696417402001109576e-01)
(-5.555555555555555802e-01, 9.887793092527931149e-01)
(-5.238095238095238360e-01, 9.983943930356182772e-01)
(-4.920634920634920917e-01, 9.983943930356182772e-01)
(-4.603174603174603474e-01, 9.887793092527931149e-01)
(-4.285714285714286031e-01, 9.696417402001109576e-01)
(-3.968253968253968589e-01, 9.411659911292082903e-01)
(-3.650793650793651146e-01, 9.036262990700658637e-01)
(-3.333333333333333703e-01, 8.573841917784963407e-01)
(-3.015873015873016261e-01, 8.028850060262247013e-01)
(-2.698412698412698818e-01, 7.406535987643307895e-01)
(-2.380952380952381375e-01, 6.712892924638946823e-01)
(-2.063492063492063933e-01, 5.954601033129761634e-01)
(-1.746031746031746490e-01, 5.138963078555496278e-01)
(-1.428571428571429047e-01, 4.273834100291765137e-01)
(-1.111111111111111605e-01, 3.367545763326852404e-01)
(-7.936507936507941618e-02, 2.428826119773333236e-01)
(-4.761904761904767192e-02, 1.466715552954981072e-01)
(-1.587301587301592765e-02, 4.904797135733926505e-02)
(1.587301587301581662e-02 , -4.904797135733926505e-02 )
(4.761904761904744987e-02 , -1.466715552954981072e-01 )
(7.936507936507930516e-02 , -2.428826119773333236e-01 )
(1.111111111111111605e-01 , -3.367545763326852404e-01 )
(1.428571428571427937e-01 , -4.273834100291765137e-01 )
(1.746031746031744269e-01 , -5.138963078555496278e-01 )
(2.063492063492062822e-01 , -5.954601033129761634e-01 )
(2.380952380952381375e-01 , -6.712892924638946823e-01 )
(2.698412698412697708e-01 , -7.406535987643307895e-01 )
(3.015873015873014040e-01 , -8.028850060262247013e-01 )
(3.333333333333332593e-01 , -8.573841917784963407e-01 )
(3.650793650793651146e-01 , -9.036262990700658637e-01 )
(3.968253968253967479e-01 , -9.411659911292082903e-01 )
(4.285714285714283811e-01 , -9.696417402001109576e-01 )
(4.603174603174602364e-01 , -9.887793092527931149e-01 )
(4.920634920634920917e-01 , -9.983943930356182772e-01 )
(5.238095238095237249e-01 , -9.983943930356182772e-01 )
(5.555555555555553582e-01 , -9.887793092527931149e-01 )
(5.873015873015872135e-01 , -9.696417402001109576e-01 )
(6.190476190476190688e-01 , -9.411659911292082903e-01 )
(6.507936507936507020e-01 , -9.036262990700658637e-01 )
(6.825396825396823353e-01 , -8.573841917784963407e-01 )
(7.142857142857141906e-01 , -8.028850060262248123e-01 )
(7.460317460317460458e-01 , -7.406535987643311225e-01 )
(7.777777777777776791e-01 , -6.712892924638949044e-01 )
(8.095238095238093123e-01 , -5.954601033129761634e-01 )
(8.412698412698411676e-01 , -5.138963078555496278e-01 )
(8.730158730158730229e-01 , -4.273834100291765692e-01 )
(9.047619047619046562e-01 , -3.367545763326854069e-01 )
(9.365079365079362894e-01 , -2.428826119773334347e-01 )
(9.682539682539681447e-01 , -1.466715552954981905e-01 )
(1.000000000000000000e+00 , -4.904797135733937608e-02 )
};
\addplot[blue, thick]
coordinates{%
(-1.000000000000000000e+00, 1.910480126056576106e-02 )
(-9.682539682539682557e-01, 5.730358861740825205e-02 )
(-9.365079365079365115e-01, 9.546957613183834623e-02 )
(-9.047619047619047672e-01, 1.335831236542430189e-01 )
(-8.730158730158730229e-01, 1.716068301135603158e-01 )
(-8.412698412698412787e-01, 2.095830487298410150e-01 )
(-8.095238095238095344e-01, 2.473038835254231238e-01 )
(-7.777777777777777901e-01, 2.851176142167071159e-01 )
(-7.460317460317460458e-01, 3.224758267350329444e-01 )
(-7.142857142857143016e-01, 3.596990576250562732e-01 )
(-6.825396825396825573e-01, 3.970177006920940777e-01 )
(-6.507936507936508130e-01, 4.337251882140734649e-01 )
(-6.190476190476190688e-01, 4.699354269344788726e-01 )
(-5.873015873015873245e-01, 5.061449726060960330e-01 )
(-5.555555555555555802e-01, 5.424180366864739256e-01 )
(-5.238095238095238360e-01, 5.777994332943252598e-01 )
(-4.920634920634920917e-01, 6.121917640704775954e-01 )
(-4.603174603174603474e-01, 6.461947354411241573e-01 )
(-4.285714285714286031e-01, 6.802408071227761255e-01 )
(-3.968253968253968589e-01, 7.143203152923122135e-01 )
(-3.650793650793651146e-01, 7.472198054354924901e-01 )
(-3.333333333333333703e-01, 7.781846581526452589e-01 )
(-3.015873015873016261e-01, 8.075002689562109648e-01 )
(-2.698412698412698818e-01, 8.360563911940114279e-01 )
(-2.380952380952381375e-01, 8.645452503912035569e-01 )
(-2.063492063492063933e-01, 8.933412462389223041e-01 )
(-1.746031746031746490e-01, 9.221868261608534567e-01 )
(-1.428571428571429047e-01, 9.476895335300393741e-01 )
(-1.111111111111111605e-01, 9.663396341032342018e-01 )
(-7.936507936507941618e-02, 9.759979482711867993e-01 )
(-4.761904761904767192e-02, 9.781913459353225315e-01 )
(-1.587301587301592765e-02, 9.766284592309375512e-01 )
(1.587301587301581662e-02 , -9.766284592309375512e-01 )
(4.761904761904744987e-02 , -9.781913459353225315e-01 )
(7.936507936507930516e-02 , -9.759979482711867993e-01 )
(1.111111111111111605e-01 , -9.663396341032342018e-01 )
(1.428571428571427937e-01 , -9.476895335300393741e-01 )
(1.746031746031744269e-01 , -9.221868261608534567e-01 )
(2.063492063492062822e-01 , -8.933412462389223041e-01 )
(2.380952380952381375e-01 , -8.645452503912035569e-01 )
(2.698412698412697708e-01 , -8.360563911940114279e-01 )
(3.015873015873014040e-01 , -8.075002689562109648e-01 )
(3.333333333333332593e-01 , -7.781846581526452589e-01 )
(3.650793650793651146e-01 , -7.472198054354924901e-01 )
(3.968253968253967479e-01 , -7.143203152923122135e-01 )
(4.285714285714283811e-01 , -6.802408071227761255e-01 )
(4.603174603174602364e-01 , -6.461947354411241573e-01 )
(4.920634920634920917e-01 , -6.121917640704775954e-01 )
(5.238095238095237249e-01 , -5.777994332943252598e-01 )
(5.555555555555553582e-01 , -5.424180366864739256e-01 )
(5.873015873015872135e-01 , -5.061449726060960330e-01 )
(6.190476190476190688e-01 , -4.699354269344788726e-01 )
(6.507936507936507020e-01 , -4.337251882140734649e-01 )
(6.825396825396823353e-01 , -3.970177006920940777e-01 )
(7.142857142857141906e-01 , -3.596990576250562732e-01 )
(7.460317460317460458e-01 , -3.224758267350329444e-01 )
(7.777777777777776791e-01 , -2.851176142167071159e-01 )
(8.095238095238093123e-01 , -2.473038835254231238e-01 )
(8.412698412698411676e-01 , -2.095830487298410150e-01 )
(8.730158730158730229e-01 , -1.716068301135603158e-01 )
(9.047619047619046562e-01 , -1.335831236542430189e-01 )
(9.365079365079362894e-01 , -9.546957613183834623e-02 )
(9.682539682539681447e-01 , -5.730358861740825205e-02 )
(1.000000000000000000e+00 , -1.910480126056576106e-02 )
};
\addplot[myred, thick] coordinates{
(-1.000000000000000000e+00, 4.098708363515574798e-02 )
(-9.682539682539682557e-01, 9.693834320311642272e-02 )
(-9.365079365079365115e-01, 1.391405445274903896e-01 )
(-9.047619047619047672e-01, 1.773303987703331863e-01 )
(-8.730158730158730229e-01, 2.152556476814951880e-01 )
(-8.412698412698412787e-01, 2.530058410455532925e-01 )
(-8.095238095238095344e-01, 2.900389627569577300e-01 )
(-7.777777777777777901e-01, 3.265543297104532838e-01 )
(-7.460317460317460458e-01, 3.626564436173872474e-01 )
(-7.142857142857143016e-01, 3.982927481635963884e-01 )
(-6.825396825396825573e-01, 4.334197607184875345e-01 )
(-6.507936507936508130e-01, 4.680142262042575330e-01 )
(-6.190476190476190688e-01, 5.020325406272567736e-01 )
(-5.873015873015873245e-01, 5.354174745245114941e-01 )
(-5.555555555555555802e-01, 5.681058211607734609e-01 )
(-5.238095238095238360e-01, 6.000246893984249041e-01 )
(-4.920634920634920917e-01, 6.310879403115613462e-01 )
(-4.603174603174603474e-01, 6.611950354403213215e-01 )
(-4.285714285714286031e-01, 6.902280097255346636e-01 )
(-3.968253968253968589e-01, 7.180465577969383784e-01 )
(-3.650793650793651146e-01, 7.444821671718426792e-01 )
(-3.333333333333333703e-01, 7.693299472549524065e-01 )
(-3.015873015873016261e-01, 7.923366743457256334e-01 )
(-2.698412698412698818e-01, 8.131839866615258039e-01 )
(-2.380952380952381375e-01, 8.314625873978167547e-01 )
(-2.063492063492063933e-01, 8.466333246944959257e-01 )
(-1.746031746031746490e-01, 8.579629300051840879e-01 )
(-1.428571428571429047e-01, 8.644169899832989667e-01 )
(-1.111111111111111605e-01, 8.645365265011860645e-01 )
(-7.936507936507941618e-02, 8.559776459397934190e-01 )
(-4.761904761904767192e-02, 8.270538694274857106e-01 )
(-1.587301587301592765e-02, 5.350014980733557302e-01 )
(1.587301587301581662e-02 , -5.350014980733557302e-01 )
(4.761904761904744987e-02 , -8.270538694274857106e-01 )
(7.936507936507930516e-02 , -8.559776459397934190e-01 )
(1.111111111111111605e-01 , -8.645365265011860645e-01 )
(1.428571428571427937e-01 , -8.644169899832989667e-01 )
(1.746031746031744269e-01 , -8.579629300051840879e-01 )
(2.063492063492062822e-01 , -8.466333246944959257e-01 )
(2.380952380952381375e-01 , -8.314625873978167547e-01 )
(2.698412698412697708e-01 , -8.131839866615258039e-01 )
(3.015873015873014040e-01 , -7.923366743457256334e-01 )
(3.333333333333332593e-01 , -7.693299472549524065e-01 )
(3.650793650793651146e-01 , -7.444821671718426792e-01 )
(3.968253968253967479e-01 , -7.180465577969383784e-01 )
(4.285714285714283811e-01 , -6.902280097255346636e-01 )
(4.603174603174602364e-01 , -6.611950354403213215e-01 )
(4.920634920634920917e-01 , -6.310879403115613462e-01 )
(5.238095238095237249e-01 , -6.000246893984249041e-01 )
(5.555555555555553582e-01 , -5.681058211607734609e-01 )
(5.873015873015872135e-01 , -5.354174745245114941e-01 )
(6.190476190476190688e-01 , -5.020325406272567736e-01 )
(6.507936507936507020e-01 , -4.680142262042575330e-01 )
(6.825396825396823353e-01 , -4.334197607184875345e-01 )
(7.142857142857141906e-01 , -3.982927481635963884e-01 )
(7.460317460317460458e-01 , -3.626564436173872474e-01 )
(7.777777777777776791e-01 , -3.265543297104532838e-01 )
(8.095238095238093123e-01 , -2.900389627569577300e-01 )
(8.412698412698411676e-01 , -2.530058410455532925e-01 )
(8.730158730158730229e-01 , -2.152556476814951880e-01 )
(9.047619047619046562e-01 , -1.773303987703331863e-01 )
(9.365079365079362894e-01 , -1.391405445274903896e-01 )
(9.682539682539681447e-01 , -9.693834320311642272e-02 )
(1.000000000000000000e+00 , -4.098708363515574798e-02 )
};
\end{axis}
\end{tikzpicture}
\caption{Experiment 3. Second-order numerical approximations for the local conservation law and the nonlocal model (in blue and red, respectively) with $\Dx=1/32$, $\lambda = 0.25$, $\delta=0.125$, $p=1$, and $u(x,0)=u_0^2(x)$.}
\label{fig: stationary shock}
\end{figure}
Secondly, we consider the initial datum
\begin{equation*}
	u_0^3(x) = 1+u_0^2(x) = 1-\sin(\pi x)
\end{equation*}
again on the interval $[-1,1]$. In the local conservation law this initial datum leads to a non-stationary shock. However, because of \Cref{thm: stationarity of discontinuities of weak solutions} we know that the corresponding solution of the nonlocal model cannot display non-stationary shocks. \Cref{fig: moving shock} clearly shows that the solution of the nonlocal model (red) is smooth where the solution of the local conservation law (blue) has a shock.
\begin{figure}[t]
\centering
\subfloat[Shock/smoothed `shock'][Shock/smoothed `shock'\\\phantom{(a) }($T=0.5$)]{
\begin{tikzpicture}
\begin{axis}[mark size=1.7pt]
\addplot+[gray, thick, dashed, mark=none] coordinates{
(-1.000000000000000000e+00, 1.049047971357339515e+00 )
(-9.682539682539682557e-01, 1.146671555295498024e+00 )
(-9.365079365079365115e-01, 1.242882611977333296e+00 )
(-9.047619047619047672e-01, 1.336754576332685351e+00 )
(-8.730158730158730229e-01, 1.427383410029176458e+00 )
(-8.412698412698412787e-01, 1.513896307855549628e+00 )
(-8.095238095238095344e-01, 1.595460103312976052e+00 )
(-7.777777777777777901e-01, 1.671289292463894682e+00 )
(-7.460317460317460458e-01, 1.740653598764331011e+00 )
(-7.142857142857143016e-01, 1.802885006026224923e+00 )
(-6.825396825396825573e-01, 1.857384191778496341e+00 )
(-6.507936507936508130e-01, 1.903626299070065642e+00 )
(-6.190476190476190688e-01, 1.941165991129208290e+00 )
(-5.873015873015873245e-01, 1.969641740200111180e+00 )
(-5.555555555555555802e-01, 1.988779309252793670e+00 )
(-5.238095238095238360e-01, 1.998394393035618499e+00 )
(-4.920634920634920917e-01, 1.998394393035618499e+00 )
(-4.603174603174603474e-01, 1.988779309252793670e+00 )
(-4.285714285714286031e-01, 1.969641740200111180e+00 )
(-3.968253968253968589e-01, 1.941165991129208068e+00 )
(-3.650793650793651146e-01, 1.903626299070065642e+00 )
(-3.333333333333333703e-01, 1.857384191778496341e+00 )
(-3.015873015873016261e-01, 1.802885006026224701e+00 )
(-2.698412698412698818e-01, 1.740653598764331011e+00 )
(-2.380952380952381375e-01, 1.671289292463894682e+00 )
(-2.063492063492063933e-01, 1.595460103312976052e+00 )
(-1.746031746031746490e-01, 1.513896307855549406e+00 )
(-1.428571428571429047e-01, 1.427383410029176680e+00 )
(-1.111111111111111605e-01, 1.336754576332685351e+00 )
(-7.936507936507941618e-02, 1.242882611977333296e+00 )
(-4.761904761904767192e-02, 1.146671555295498024e+00 )
(-1.587301587301592765e-02, 1.049047971357339293e+00 )
(1.587301587301581662e-02 , 9.509520286426607072e-01 )
(4.761904761904744987e-02 , 8.533284447045019760e-01 )
(7.936507936507930516e-02 , 7.571173880226665931e-01 )
(1.111111111111111605e-01 , 6.632454236673146486e-01 )
(1.428571428571427937e-01 , 5.726165899708235418e-01 )
(1.746031746031744269e-01 , 4.861036921444504277e-01 )
(2.063492063492062822e-01 , 4.045398966870239477e-01 )
(2.380952380952381375e-01 , 3.287107075361052622e-01 )
(2.698412698412697708e-01 , 2.593464012356689885e-01 )
(3.015873015873014040e-01 , 1.971149939737752432e-01 )
(3.333333333333332593e-01 , 1.426158082215037148e-01 )
(3.650793650793651146e-01 , 9.637370092993424731e-02 )
(3.968253968253967479e-01 , 5.883400887079182073e-02 )
(4.285714285714283811e-01 , 3.035825979988916040e-02 )
(4.603174603174602364e-01 , 1.122069074720683304e-02 )
(4.920634920634920917e-01 , 1.605606964381636487e-03 )
(5.238095238095237249e-01 , 1.605606964381616104e-03 )
(5.555555555555553582e-01 , 1.122069074720681396e-02 )
(5.873015873015872135e-01 , 3.035825979988911530e-02 )
(6.190476190476190688e-01 , 5.883400887079179298e-02 )
(6.507936507936507020e-01 , 9.637370092993420567e-02 )
(6.825396825396823353e-01 , 1.426158082215036316e-01 )
(7.142857142857141906e-01 , 1.971149939737751322e-01 )
(7.460317460317460458e-01 , 2.593464012356688775e-01 )
(7.777777777777776791e-01 , 3.287107075361051511e-01 )
(8.095238095238093123e-01 , 4.045398966870238366e-01 )
(8.412698412698411676e-01 , 4.861036921444503722e-01 )
(8.730158730158730229e-01 , 5.726165899708235418e-01 )
(9.047619047619046562e-01 , 6.632454236673145376e-01 )
(9.365079365079362894e-01 , 7.571173880226665931e-01 )
(9.682539682539681447e-01 , 8.533284447045017540e-01 )
(1.000000000000000000e+00 , 9.509520286426607072e-01 )
};
\addplot[blue, thick] coordinates{
(-1.000000000000000000e+00, 4.229570752315855442e-01 )
(-9.682539682539682557e-01, 4.582734866797598117e-01 )
(-9.365079365079365115e-01, 4.940086852120489924e-01 )
(-9.047619047619047672e-01, 5.301178918393130157e-01 )
(-8.730158730158730229e-01, 5.665605620918612306e-01 )
(-8.412698412698412787e-01, 6.032997196889977598e-01 )
(-8.095238095238095344e-01, 6.403014254543313299e-01 )
(-7.777777777777777901e-01, 6.775343669829447801e-01 )
(-7.460317460317460458e-01, 7.149695697270108674e-01 )
(-7.142857142857143016e-01, 7.525802525784237584e-01 )
(-6.825396825396825573e-01, 7.903418450211161517e-01 )
(-6.507936507936508130e-01, 8.282320684966588686e-01 )
(-6.190476190476190688e-01, 8.662308046024258523e-01 )
(-5.873015873015873245e-01, 9.043195397851691020e-01 )
(-5.555555555555555802e-01, 9.424808703433791290e-01 )
(-5.238095238095238360e-01, 9.806998321623815462e-01 )
(-4.920634920634920917e-01, 1.018969867637711690e+00 )
(-4.603174603174603474e-01, 1.057306466667869094e+00 )
(-4.285714285714286031e-01, 1.095718180467338376e+00 )
(-3.968253968253968589e-01, 1.133924212967731826e+00 )
(-3.650793650793651146e-01, 1.171842655639324660e+00 )
(-3.333333333333333703e-01, 1.209612645073592896e+00 )
(-3.015873015873016261e-01, 1.247360784770531694e+00 )
(-2.698412698412698818e-01, 1.285152725884132874e+00 )
(-2.380952380952381375e-01, 1.323039161089936577e+00 )
(-2.063492063492063933e-01, 1.360560086136771041e+00 )
(-1.746031746031746490e-01, 1.397421666957434638e+00 )
(-1.428571428571429047e-01, 1.433737228403053843e+00 )
(-1.111111111111111605e-01, 1.469774814906254168e+00 )
(-7.936507936507941618e-02, 1.505759467445316080e+00 )
(-4.761904761904767192e-02, 1.541794011666618225e+00 )
(-1.587301587301592765e-02, 1.578012394437699584e+00 )
(1.587301587301581662e-02 , 1.613900816078657430e+00 )
(4.761904761904744987e-02 , 1.648504674247748714e+00 )
(7.936507936507930516e-02 , 1.681612230425295618e+00 )
(1.111111111111111605e-01 , 1.713444548916521537e+00 )
(1.428571428571427937e-01 , 1.744510768572422821e+00 )
(1.746031746031744269e-01 , 1.775312875911324806e+00 )
(2.063492063492062822e-01 , 1.806121510369664396e+00 )
(2.380952380952381375e-01 , 1.837116483305135439e+00 )
(2.698412698412697708e-01 , 1.868676048946453561e+00 )
(3.015873015873014040e-01 , 1.900927058553014026e+00 )
(3.333333333333332593e-01 , 1.929759433082039788e+00 )
(3.650793650793651146e-01 , 1.950621613257762954e+00 )
(3.968253968253967479e-01 , 1.959837016288036970e+00 )
(4.285714285714283811e-01 , 1.958625705261013206e+00 )
(4.603174603174602364e-01 , 1.928067967775561442e+00 )
(4.920634920634920917e-01 , 1.626677470761314748e+00 )
(5.238095238095237249e-01 , 4.086760995160447774e-01 )
(5.555555555555553582e-01 , 1.248004090375605737e-02 )
(5.873015873015872135e-01 , 2.271109233082450538e-02 )
(6.190476190476190688e-01 , 3.999808720544177632e-02 )
(6.507936507936507020e-01 , 6.039871181566273950e-02 )
(6.825396825396823353e-01 , 8.331112275789992472e-02 )
(7.142857142857141906e-01 , 1.082915912433426858e-01 )
(7.460317460317460458e-01 , 1.350018025338567684e-01 )
(7.777777777777776791e-01 , 1.631803622501896145e-01 )
(8.095238095238093123e-01 , 1.926202721258742923e-01 )
(8.412698412698411676e-01 , 2.231541628233295604e-01 )
(8.730158730158730229e-01 , 2.546443111174374607e-01 )
(9.047619047619046562e-01 , 2.869756498352594720e-01 )
(9.365079365079362894e-01 , 3.200507402612196017e-01 )
(9.682539682539681447e-01 , 3.537860730947780530e-01 )
(1.000000000000000000e+00 , 3.881092923843938447e-01 )
};
\addplot[myred, thick] coordinates{
(-1.000000000000000000e+00, 4.693520067987828059e-01 )
(-9.682539682539682557e-01, 5.057394764578029278e-01 )
(-9.365079365079365115e-01, 5.423648579219573218e-01 )
(-9.047619047619047672e-01, 5.791905426571388738e-01 )
(-8.730158730158730229e-01, 6.161818566761171345e-01 )
(-8.412698412698412787e-01, 6.533064523741614060e-01 )
(-8.095238095238095344e-01, 6.905338388944419759e-01 )
(-7.777777777777777901e-01, 7.278350002149157616e-01 )
(-7.460317460317460458e-01, 7.651820423920485670e-01 )
(-7.142857142857143016e-01, 8.025478691132170717e-01 )
(-6.825396825396825573e-01, 8.399056365357613352e-01 )
(-6.507936507936508130e-01, 8.772291257170142798e-01 )
(-6.190476190476190688e-01, 9.144921796398363112e-01 )
(-5.873015873015873245e-01, 9.516683826568561599e-01 )
(-5.555555555555555802e-01, 9.887309078405355223e-01 )
(-5.238095238095238360e-01, 1.025652249712955655e+00 )
(-4.920634920634920917e-01, 1.062403952396466567e+00 )
(-4.603174603174603474e-01, 1.098956333627931992e+00 )
(-4.285714285714286031e-01, 1.135278172928963603e+00 )
(-3.968253968253968589e-01, 1.171336359300599650e+00 )
(-3.650793650793651146e-01, 1.207095492651293123e+00 )
(-3.333333333333333703e-01, 1.242517420650120386e+00 )
(-3.015873015873016261e-01, 1.277560695487697240e+00 )
(-2.698412698412698818e-01, 1.312179923132915205e+00 )
(-2.380952380952381375e-01, 1.346324981975998369e+00 )
(-2.063492063492063933e-01, 1.379940068337526693e+00 )
(-1.746031746031746490e-01, 1.412962509092855257e+00 )
(-1.428571428571429047e-01, 1.445321283384755962e+00 )
(-1.111111111111111605e-01, 1.476935140043013739e+00 )
(-7.936507936507941618e-02, 1.507710154525860657e+00 )
(-4.761904761904767192e-02, 1.537536584024647768e+00 )
(-1.587301587301592765e-02, 1.566284613673671000e+00 )
(1.587301587301581662e-02 , 1.593798618291138336e+00 )
(4.761904761904744987e-02 , 1.619889267576121838e+00 )
(7.936507936507930516e-02 , 1.644322329898152635e+00 )
(1.111111111111111605e-01 , 1.666802381388507559e+00 )
(1.428571428571427937e-01 , 1.686948677100014615e+00 )
(1.746031746031744269e-01 , 1.704258065523477583e+00 )
(2.063492063492062822e-01 , 1.718045904786229761e+00 )
(2.380952380952381375e-01 , 1.727349642095845628e+00 )
(2.698412698412697708e-01 , 1.730789421613658785e+00 )
(3.015873015873014040e-01 , 1.726372368980764627e+00 )
(3.333333333333332593e-01 , 1.710726632934388647e+00 )
(3.650793650793651146e-01 , 1.678584364878217494e+00 )
(3.968253968253967479e-01 , 1.621382395285567490e+00 )
(4.285714285714283811e-01 , 1.525765902414248920e+00 )
(4.603174603174602364e-01 , 1.373938264245497587e+00 )
(4.920634920634920917e-01 , 1.151395099538155664e+00 )
(5.238095238095237249e-01 , 8.671993856443034332e-01 )
(5.555555555555553582e-01 , 5.713667416454890624e-01 )
(5.873015873015872135e-01 , 3.320652894136856825e-01 )
(6.190476190476190688e-01 , 1.863119340884708819e-01 )
(6.507936507936507020e-01 , 1.254457730825036255e-01 )
(6.825396825396823353e-01 , 1.180192569448568207e-01 )
(7.142857142857141906e-01 , 1.357217182097513741e-01 )
(7.460317460317460458e-01 , 1.629052538673207706e-01 )
(7.777777777777776791e-01 , 1.934875766058681834e-01 )
(8.095238095238093123e-01 , 2.256317748989613570e-01 )
(8.412698412698411676e-01 , 2.587583292215934749e-01 )
(8.730158730158730229e-01 , 2.926208021756954714e-01 )
(9.047619047619046562e-01 , 3.270831797717659306e-01 )
(9.365079365079362894e-01 , 3.620535415904750365e-01 )
(9.682539682539681447e-01 , 3.974603000234948613e-01 )
(1.000000000000000000e+00 , 4.332436954279201280e-01 )
};
\end{axis}
\end{tikzpicture}
\label{fig: moving shock a}
}%
\subfloat[Shock/smoothed `shock'][Shock/smoothed `shock'\\\phantom{(b) }($T=1.5$)]{
\begin{tikzpicture}
\begin{axis}[mark size=1.7pt]
\addplot+[gray, thick, dashed, mark=none] coordinates{
(-1.000000000000000000e+00, 1.049047971357339515e+00 )
(-9.682539682539682557e-01, 1.146671555295498024e+00 )
(-9.365079365079365115e-01, 1.242882611977333296e+00 )
(-9.047619047619047672e-01, 1.336754576332685351e+00 )
(-8.730158730158730229e-01, 1.427383410029176458e+00 )
(-8.412698412698412787e-01, 1.513896307855549628e+00 )
(-8.095238095238095344e-01, 1.595460103312976052e+00 )
(-7.777777777777777901e-01, 1.671289292463894682e+00 )
(-7.460317460317460458e-01, 1.740653598764331011e+00 )
(-7.142857142857143016e-01, 1.802885006026224923e+00 )
(-6.825396825396825573e-01, 1.857384191778496341e+00 )
(-6.507936507936508130e-01, 1.903626299070065642e+00 )
(-6.190476190476190688e-01, 1.941165991129208290e+00 )
(-5.873015873015873245e-01, 1.969641740200111180e+00 )
(-5.555555555555555802e-01, 1.988779309252793670e+00 )
(-5.238095238095238360e-01, 1.998394393035618499e+00 )
(-4.920634920634920917e-01, 1.998394393035618499e+00 )
(-4.603174603174603474e-01, 1.988779309252793670e+00 )
(-4.285714285714286031e-01, 1.969641740200111180e+00 )
(-3.968253968253968589e-01, 1.941165991129208068e+00 )
(-3.650793650793651146e-01, 1.903626299070065642e+00 )
(-3.333333333333333703e-01, 1.857384191778496341e+00 )
(-3.015873015873016261e-01, 1.802885006026224701e+00 )
(-2.698412698412698818e-01, 1.740653598764331011e+00 )
(-2.380952380952381375e-01, 1.671289292463894682e+00 )
(-2.063492063492063933e-01, 1.595460103312976052e+00 )
(-1.746031746031746490e-01, 1.513896307855549406e+00 )
(-1.428571428571429047e-01, 1.427383410029176680e+00 )
(-1.111111111111111605e-01, 1.336754576332685351e+00 )
(-7.936507936507941618e-02, 1.242882611977333296e+00 )
(-4.761904761904767192e-02, 1.146671555295498024e+00 )
(-1.587301587301592765e-02, 1.049047971357339293e+00 )
(1.587301587301581662e-02 , 9.509520286426607072e-01 )
(4.761904761904744987e-02 , 8.533284447045019760e-01 )
(7.936507936507930516e-02 , 7.571173880226665931e-01 )
(1.111111111111111605e-01 , 6.632454236673146486e-01 )
(1.428571428571427937e-01 , 5.726165899708235418e-01 )
(1.746031746031744269e-01 , 4.861036921444504277e-01 )
(2.063492063492062822e-01 , 4.045398966870239477e-01 )
(2.380952380952381375e-01 , 3.287107075361052622e-01 )
(2.698412698412697708e-01 , 2.593464012356689885e-01 )
(3.015873015873014040e-01 , 1.971149939737752432e-01 )
(3.333333333333332593e-01 , 1.426158082215037148e-01 )
(3.650793650793651146e-01 , 9.637370092993424731e-02 )
(3.968253968253967479e-01 , 5.883400887079182073e-02 )
(4.285714285714283811e-01 , 3.035825979988916040e-02 )
(4.603174603174602364e-01 , 1.122069074720683304e-02 )
(4.920634920634920917e-01 , 1.605606964381636487e-03 )
(5.238095238095237249e-01 , 1.605606964381616104e-03 )
(5.555555555555553582e-01 , 1.122069074720681396e-02 )
(5.873015873015872135e-01 , 3.035825979988911530e-02 )
(6.190476190476190688e-01 , 5.883400887079179298e-02 )
(6.507936507936507020e-01 , 9.637370092993420567e-02 )
(6.825396825396823353e-01 , 1.426158082215036316e-01 )
(7.142857142857141906e-01 , 1.971149939737751322e-01 )
(7.460317460317460458e-01 , 2.593464012356688775e-01 )
(7.777777777777776791e-01 , 3.287107075361051511e-01 )
(8.095238095238093123e-01 , 4.045398966870238366e-01 )
(8.412698412698411676e-01 , 4.861036921444503722e-01 )
(8.730158730158730229e-01 , 5.726165899708235418e-01 )
(9.047619047619046562e-01 , 6.632454236673145376e-01 )
(9.365079365079362894e-01 , 7.571173880226665931e-01 )
(9.682539682539681447e-01 , 8.533284447045017540e-01 )
(1.000000000000000000e+00 , 9.509520286426607072e-01 )
};
\addplot[blue, thick] coordinates{
(-1.000000000000000000e+00, 1.283307274432809120e+00 )
(-9.682539682539682557e-01, 1.300401437483755585e+00 )
(-9.365079365079365115e-01, 1.317439362823610516e+00 )
(-9.047619047619047672e-01, 1.334422139850009348e+00 )
(-8.730158730158730229e-01, 1.351351754730386823e+00 )
(-8.412698412698412787e-01, 1.368234202318624249e+00 )
(-8.095238095238095344e-01, 1.385088535586390224e+00 )
(-7.777777777777777901e-01, 1.401938251823800385e+00 )
(-7.460317460317460458e-01, 1.418803828071715678e+00 )
(-7.142857142857143016e-01, 1.435781315348927833e+00 )
(-6.825396825396825573e-01, 1.453381347984028515e+00 )
(-6.507936507936508130e-01, 1.474308302841368157e+00 )
(-6.190476190476190688e-01, 1.497948473366636390e+00 )
(-5.873015873015873245e-01, 1.499965247789520451e+00 )
(-5.555555555555555802e-01, 1.465256785983745313e+00 )
(-5.238095238095238360e-01, 1.263821236621643918e+00 )
(-4.920634920634920917e-01, 7.454567758162652380e-01 )
(-4.603174603174603474e-01, 5.130384578190427813e-01 )
(-4.285714285714286031e-01, 4.984365523586700153e-01 )
(-3.968253968253968589e-01, 5.097578821509635150e-01 )
(-3.650793650793651146e-01, 5.289001718137958941e-01 )
(-3.333333333333333703e-01, 5.473198950628628268e-01 )
(-3.015873015873016261e-01, 5.643488117478097088e-01 )
(-2.698412698412698818e-01, 5.812349196652001337e-01 )
(-2.380952380952381375e-01, 5.981235308951202700e-01 )
(-2.063492063492063933e-01, 6.150361058321516250e-01 )
(-1.746031746031746490e-01, 6.319790861001449711e-01 )
(-1.428571428571429047e-01, 6.489481350997494680e-01 )
(-1.111111111111111605e-01, 6.659381364504824985e-01 )
(-7.936507936507941618e-02, 6.829468723527688523e-01 )
(-4.761904761904767192e-02, 6.999738229147658419e-01 )
(-1.587301587301592765e-02, 7.170184458527424365e-01 )
(1.587301587301581662e-02 , 7.340796887343541410e-01 )
(4.761904761904744987e-02 , 7.511562740741741795e-01 )
(7.936507936507930516e-02 , 7.682470090470534529e-01 )
(1.111111111111111605e-01 , 7.853508657252976644e-01 )
(1.428571428571427937e-01 , 8.024669250940101861e-01 )
(1.746031746031744269e-01 , 8.195943174237373396e-01 )
(2.063492063492062822e-01 , 8.367322067966065369e-01 )
(2.380952380952381375e-01 , 8.538798057436478572e-01 )
(2.698412698412697708e-01 , 8.710363960031546338e-01 )
(3.015873015873014040e-01 , 8.882013453201159647e-01 )
(3.333333333333332593e-01 , 9.053741222376870956e-01 )
(3.650793650793651146e-01 , 9.225543138552635414e-01 )
(3.968253968253967479e-01 , 9.397416491632761826e-01 )
(4.285714285714283811e-01 , 9.569360274017230106e-01 )
(4.603174603174602364e-01 , 9.741375486934655825e-01 )
(4.920634920634920917e-01 , 9.913465426353189613e-01 )
(5.238095238095237249e-01 , 1.008563591097535017e+00 )
(5.555555555555553582e-01 , 1.025789552150159878e+00 )
(5.873015873015872135e-01 , 1.043023776701561767e+00 )
(6.190476190476190688e-01 , 1.060245528604918608e+00 )
(6.507936507936507020e-01 , 1.077444159667310775e+00 )
(6.825396825396823353e-01 , 1.094622447282587618e+00 )
(7.142857142857141906e-01 , 1.111784976586207563e+00 )
(7.460317460317460458e-01 , 1.128937108509496579e+00 )
(7.777777777777776791e-01 , 1.146083834300100968e+00 )
(8.095238095238093123e-01 , 1.163228793290435714e+00 )
(8.412698412698411676e-01 , 1.180373732553916311e+00 )
(8.730158730158730229e-01 , 1.197519647387572439e+00 )
(9.047619047619046562e-01 , 1.214668453254861946e+00 )
(9.365079365079362894e-01 , 1.231823099729365767e+00 )
(9.682539682539681447e-01 , 1.248987574430430403e+00 )
(1.000000000000000000e+00 , 1.266161587515144182e+00 )
};
\addplot[myred, thick] coordinates{
(-1.000000000000000000e+00, 1.300983661740329200e+00 )
(-9.682539682539682557e-01, 1.314694714773623652e+00 )
(-9.365079365079365115e-01, 1.327269351821331433e+00 )
(-9.047619047619047672e-01, 1.338188075179221226e+00 )
(-8.730158730158730229e-01, 1.346705806262539440e+00 )
(-8.412698412698412787e-01, 1.351776992370651298e+00 )
(-8.095238095238095344e-01, 1.352034816747087920e+00 )
(-7.777777777777777901e-01, 1.345667708715729383e+00 )
(-7.460317460317460458e-01, 1.330342173803381822e+00 )
(-7.142857142857143016e-01, 1.303580138532665966e+00 )
(-6.825396825396825573e-01, 1.263029919699073744e+00 )
(-6.507936507936508130e-01, 1.207232070110383138e+00 )
(-6.190476190476190688e-01, 1.136583576500507231e+00 )
(-5.873015873015873245e-01, 1.054159719653054994e+00 )
(-5.555555555555555802e-01, 9.658246254803812825e-01 )
(-5.238095238095238360e-01, 8.792427626204092350e-01 )
(-4.920634920634920917e-01, 8.017852746057011482e-01 )
(-4.603174603174603474e-01, 7.386661216967826959e-01 )
(-4.285714285714286031e-01, 6.920344816744814498e-01 )
(-3.968253968253968589e-01, 6.612413809459850578e-01 )
(-3.650793650793651146e-01, 6.438975504994934873e-01 )
(-3.333333333333333703e-01, 6.370082727749214468e-01 )
(-3.015873015873016261e-01, 6.377297021268236721e-01 )
(-2.698412698412698818e-01, 6.438012264621166825e-01 )
(-2.380952380952381375e-01, 6.535256795732037327e-01 )
(-2.063492063492063933e-01, 6.656849891908277472e-01 )
(-1.746031746031746490e-01, 6.794606224565025165e-01 )
(-1.428571428571429047e-01, 6.943039779257489386e-01 )
(-1.111111111111111605e-01, 7.098513210746054680e-01 )
(-7.936507936507941618e-02, 7.258631813232232499e-01 )
(-4.761904761904767192e-02, 7.421817272615858219e-01 )
(-1.587301587301592765e-02, 7.587026423412643883e-01 )
(1.587301587301581662e-02 , 7.753566804970630777e-01 )
(4.761904761904744987e-02 , 7.920975644992621634e-01 )
(7.936507936507930516e-02 , 8.088940910106370286e-01 )
(1.111111111111111605e-01 , 8.257249792355874085e-01 )
(1.428571428571427937e-01 , 8.425754989091427394e-01 )
(1.746031746031744269e-01 , 8.594352550860018258e-01 )
(2.063492063492062822e-01 , 8.762967137128947481e-01 )
(2.380952380952381375e-01 , 8.931542374138861096e-01 )
(2.698412698412697708e-01 , 9.100034950976017001e-01 )
(3.015873015873014040e-01 , 9.268409575222822383e-01 )
(3.333333333333332593e-01 , 9.436636262932021069e-01 )
(3.650793650793651146e-01 , 9.604688321451821942e-01 )
(3.968253968253967479e-01 , 9.772540918455583370e-01 )
(4.285714285714283811e-01 , 9.940170060656572382e-01 )
(4.603174603174602364e-01 , 1.010755182003458374e+00 )
(4.920634920634920917e-01 , 1.027466170146559499e+00 )
(5.238095238095237249e-01 , 1.044147407639676395e+00 )
(5.555555555555553582e-01 , 1.060796161082687261e+00 )
(5.873015873015872135e-01 , 1.077409461103115973e+00 )
(6.190476190476190688e-01 , 1.093984018947001369e+00 )
(6.507936507936507020e-01 , 1.110516110884716579e+00 )
(6.825396825396823353e-01 , 1.127001408826770135e+00 )
(7.142857142857141906e-01 , 1.143434722993906361e+00 )
(7.460317460317460458e-01 , 1.159809604300692909e+00 )
(7.777777777777776791e-01 , 1.176117722655933395e+00 )
(8.095238095238093123e-01 , 1.192347897576518712e+00 )
(8.412698412698411676e-01 , 1.208484585776947284e+00 )
(8.730158730158730229e-01 , 1.224505542842956984e+00 )
(9.047619047619046562e-01 , 1.240378238055289906e+00 )
(9.365079365079362894e-01 , 1.256054393347656895e+00 )
(9.682539682539681447e-01 , 1.271461826545886797e+00 )
(1.000000000000000000e+00 , 1.286492249992628345e+00 )
};
\end{axis}
\end{tikzpicture}
\label{fig: moving shock b}
}%
\caption{Experiment 3. Second-order numerical approximations for the local conservation law and the nonlocal model (in blue and red, respectively) with $\Dx=1/32$, $\lambda = 0.25$, $\delta=0.125$, $p=1$, and $u(x,0)=u_0^3(x)$.}
\label{fig: moving shock}
\end{figure}
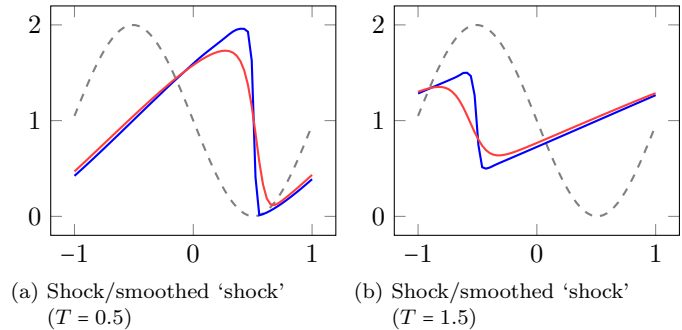

\subsection{Experiment 4: Comparison of different numerical flux functions}
In this numerical experiment we want to investigate how the (stationary) shock formation in the nonlocal model depends on the choice of numerical flux $g$ in the model. To that end, we compare the Godunov flux, the Engquist--Osher flux and the Lax--Friedrichs flux. First, we consider the Riemann problem for the nonlocal model, using the initial datum
\begin{equation*}
	u_0^4(x) = \begin{cases}
		\phantom{-}1 & \text{if }x<0,\\
		-1 & \text{if }x>0.
	\end{cases}
\end{equation*}
\Cref{fig: Riemann problem} shows numerical solutions computed with the second-order scheme~\eqref{num scheme} for the local conservation law (blue) and for the nonlocal model (red) at $T=1$.
\begin{figure}[t]
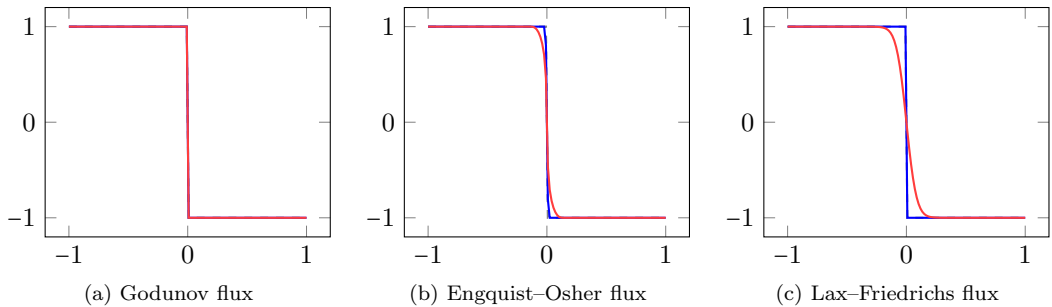

\centering
\subfloat[Godunov flux]{

\label{Godunov Riemann problem}
}%
\subfloat[Engquist--Osher flux]{
%
\label{EO flux Riemann problem}
}%
\subfloat[Lax--Friedrichs flux]{
%
\label{LxF Riemann problem}
}%
\caption{Experiment 4. Second-order numerical approximations for the local conservation law and the nonlocal model (in blue and red respectively) with $\Dx=2/128$, $\lambda = 0.8$, $\delta=0.125$, $p=0$, $T=1$, and $u(x,0)=u_0^4(x)$.}
\label{fig: Riemann problem}
\end{figure}
Note that the entropy solution of the local conservation law is the stationary shock given by $u(x,t)=u_0^4(x)$. We observe that in the case of the Godunov flux the entropy solution of the nonlocal model also is the stationary shock $u(x,t)=u_0^4(x)$, in case of the Engquist--Osher flux it is a stationary shock centered at $x=0$, but not equal to $u_0^4$, and in case of the Lax--Friedrich flux the entropy solution of the nonlocal model appears to be smooth.

In fact, by going back to the entropy condition in \Cref{def: entropy solution}, we can prove that the Riemann problem for the nonlocal model with general initial datum
\begin{equation*}
	u_0(x) = \begin{cases}
		u_L & \text{if }x<0,\\
		u_R & \text{if }x>0,
	\end{cases}
\end{equation*}
has the entropy solution $u(x,t)=u_0(x)$ if and only if $f(u_L)=f(u_R)$ and the numerical flux satisfies $g(u_L,u_R) = f(u_L)\,(=f(u_R))$, which is satisfied by the Godunov flux and the upwind and downwind fluxes, but not by the Engquist--Osher or Lax--Friedrichs flux.
\begin{figure}[t]
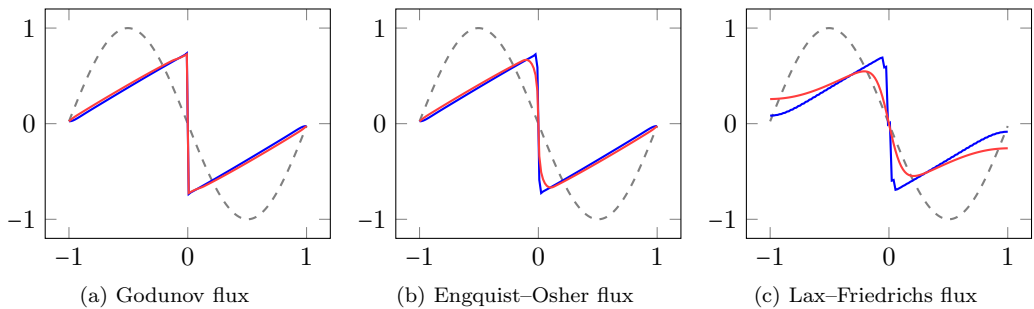

\centering
\subfloat[Godunov flux]{

}%
\caption{Experiment 4. Second-order numerical approximations for the local conservation law and the nonlocal model (in blue and red, respectively) with $\Dx=2/128$, $\lambda = 0.8$, $\delta=0.125$, $p=0$, $T=1$, and $u(x,0) = u_0^2(x)$.}
\label{fig: flux comparison sine}
\end{figure}
\begin{table}[t]
\centering
%
%
%
\subfloat[Godunov flux--nonlocal]{
\begin{tabular}{@{}rcc@{}}
  \toprule
  \multicolumn{1}{c}{$n$} & $\Lone$ error & $\Lone$ OOC\\
  \midrule
	 $  8$&  $\num{6.998e-02}$ & -- \\
	 $ 16$&  $\num{9.977e-03}$ & $2.81$ \\
	 $ 32$&  $\num{2.160e-03}$ & $2.21$ \\
	 $ 64$&  $\num{4.492e-04}$ & $2.27$ \\
	 $128$&  $\num{1.346e-04}$ & $1.74$ \\
	 $256$&  $\num{5.922e-05}$ & $1.18$ \\
	 $512$&  $\num{2.753e-05}$ & $1.10$ \\
  \bottomrule
\end{tabular}
}%
\subfloat[Engquist--Osher flux--nonlocal]{
\begin{tabular}{@{}rcc@{}}
  \toprule
  \multicolumn{1}{c}{$n$} & $\Lone$ error & $\Lone$ OOC\\
  \midrule
	$  8$ &  $\num{1.186e-01}$ & -- \\
	$ 16$ &  $\num{2.497e-02}$ & $2.25$ \\
	$ 32$ &  $\num{4.234e-03}$ & $2.56$ \\
	$ 64$ &  $\num{6.805e-04}$ & $2.64$ \\
	$128$ &  $\num{2.245e-04}$ & $1.60$ \\
	$256$ &  $\num{1.061e-04}$ & $1.08$ \\
	$512$ &  $\num{3.550e-05}$ & $1.58$ \\
  \bottomrule
\end{tabular}
}%
\subfloat[Lax--Friedrichs flux--nonlocal]{
\begin{tabular}{@{}rcc@{}}
  \toprule
  \multicolumn{1}{c}{$n$} & $\Lone$ error & $\Lone$ OOC\\
  \midrule
	$  8$ &  $\num{1.327e+00}$ & -- \\
	$ 16$ &  $\num{6.400e-02}$ & $4.37$ \\
	$ 32$ &  $\num{1.225e-02}$ & $2.39$ \\
	$ 64$ &  $\num{3.768e-03}$ & $1.70$ \\
	$128$ &  $\num{1.573e-03}$ & $1.26$ \\
	$256$ &  $\num{6.687e-04}$ & $1.23$ \\
	$512$ &  $\num{2.248e-04}$ & $1.57$ \\
  \bottomrule
\end{tabular}
}%
\caption{Experiment 4. $\Lone$ errors and observed order of convergence of the second-order method for 
the nonlocal model 
with $\lambda=0.8$, $\delta = 0.125$, $T=1$, $u(x,0)=u_0^2(x)$, and 
$p=0$.}
\label{tbl: Errors and convergence rates experiment 4}
\end{table}

\Cref{fig: flux comparison sine} again shows numerical solutions computed with the second-order scheme for the local conservation law (blue) and for the nonlocal model (red) using the initial datum $u_0^2$ and different numerical flux functions. \Cref{tbl: Errors and convergence rates experiment 4} shows the corresponding $\Lone$ errors against an approximation on a very fine grid ($n=1024$) and the observed order of convergence. We observe a reduction in the convergence rate, which is to be expected in the presence of (stationary) shocks.

We want to highlight that Experiment 3 suggests that the nonlocal model only exhibits a (stationary) shock if the corresponding local conservation law exhibits a stationary shock. Experiment 4 on the other hand suggests that the converse is not true.
Even if the local conservation law exhibits a stationary shock the nonlocal model might not.

\subsection{Experiment 5: Asymptotic compatibility}
Lastly, we want to numerically  verify that if $(\delta,\Dx)\to (0,0)$ the numerical approximation given by the second-order scheme for the nonlocal equation converges towards the solution of the local equation. To that end, we consider the initial datum $u_0^1$ from Experiment 1 and take $\delta=3\Dx$. \Cref{fig: Experiment 5} shows numerical solutions calculated with the second-order method for the nonlocal model (red) for various values of $\Dx$ and for the local conservation law (blue) at $T=0.5$. \Cref{tbl: Errors and convergence rates experiment 5} shows the corresponding $\Lone$ errors against an approximation of the entropy solution of the local conservation law with $n=1024$. Both \Cref{fig: Experiment 5} and \Cref{tbl: Errors and convergence rates experiment 5} clearly show that the second-order scheme for the nonlocal equation is asymptotically compatible, i.e., it converges towards the entropy solution of the local conservation law -- albeit apparently at a first-order rate.

\begin{figure}[t]
\centering
\begin{tikzpicture}
\begin{axis}[mark size=1.7pt]
\addplot+[gray, thick, dashed, mark=none] coordinates{
(0.000000000000000000e+00 , 5.969195893703570510e-01)
(6.666666666666666574e-02 , 7.760036392076995471e-01)
(1.333333333333333315e-01 , 9.130686369548897119e-01)
(2.000000000000000111e-01 , 9.872476792022164549e-01)
(2.666666666666666630e-01 , 9.872476792022162329e-01)
(3.333333333333333148e-01 , 9.130686369548897119e-01)
(4.000000000000000222e-01 , 7.760036392076994360e-01)
(4.666666666666666741e-01 , 5.969195893703576061e-01)
(5.333333333333333259e-01 , 4.030804106296423939e-01)
(5.999999999999999778e-01 , 2.239963607923014521e-01)
(6.666666666666666297e-01 , 8.693136304511028811e-02)
(7.333333333333332815e-01 , 1.275232079778287897e-02)
(8.000000000000000444e-01 , 1.275232079778465533e-02)
(8.666666666666666963e-01 , 8.693136304511028811e-02)
(9.333333333333333481e-01 , 2.239963607922996758e-01)
(1.000000000000000000e+00 , 4.030804106296423939e-01)
};
\addplot[myred, thick]
table{
0.0000 0.2464
0.0323 0.2832
0.0645 0.3203
0.0968 0.3577
0.1290 0.3951
0.1613 0.4325
0.1935 0.4697
0.2258 0.5067
0.2581 0.5433
0.2903 0.5794
0.3226 0.6149
0.3548 0.6496
0.3871 0.6832
0.4194 0.7157
0.4516 0.7465
0.4839 0.7754
0.5161 0.8017
0.5484 0.8245
0.5806 0.8422
0.6129 0.8521
0.6452 0.8503
0.6774 0.8281
0.7097 0.7612
0.7419 0.6070
0.7742 0.3588
0.8065 0.1586
0.8387 0.0830
0.8710 0.0816
0.9032 0.1071
0.9355 0.1398
0.9677 0.1744
1.0000 0.2101
};
\addplot[myred!75!blue,thick]
table{
0.0000 0.2247
0.0159 0.2427
0.0317 0.2608
0.0476 0.2791
0.0635 0.2975
0.0794 0.3160
0.0952 0.3346
0.1111 0.3533
0.1270 0.3720
0.1429 0.3908
0.1587 0.4096
0.1746 0.4284
0.1905 0.4472
0.2063 0.4661
0.2222 0.4849
0.2381 0.5036
0.2540 0.5224
0.2698 0.5410
0.2857 0.5596
0.3016 0.5782
0.3175 0.5966
0.3333 0.6149
0.3492 0.6331
0.3651 0.6511
0.3810 0.6690
0.3968 0.6866
0.4127 0.7041
0.4286 0.7213
0.4444 0.7383
0.4603 0.7550
0.4762 0.7713
0.4921 0.7873
0.5079 0.8028
0.5238 0.8179
0.5397 0.8324
0.5556 0.8463
0.5714 0.8594
0.5873 0.8716
0.6032 0.8828
0.6190 0.8926
0.6349 0.9006
0.6508 0.9064
0.6667 0.9087
0.6825 0.9064
0.6984 0.8957
0.7143 0.8671
0.7302 0.7977
0.7460 0.6427
0.7619 0.3865
0.7778 0.1599
0.7937 0.0551
0.8095 0.0304
0.8254 0.0352
0.8413 0.0471
0.8571 0.0606
0.8730 0.0750
0.8889 0.0900
0.9048 0.1056
0.9206 0.1216
0.9365 0.1380
0.9524 0.1548
0.9683 0.1719
0.9841 0.1893
1.0000 0.2069
};
\addplot[myred!50!blue, thick]
table{
0.0000 0.2137
0.0079 0.2226
0.0157 0.2315
0.0236 0.2405
0.0315 0.2495
0.0394 0.2585
0.0472 0.2676
0.0551 0.2767
0.0630 0.2859
0.0709 0.2951
0.0787 0.3043
0.0866 0.3136
0.0945 0.3228
0.1024 0.3322
0.1102 0.3415
0.1181 0.3508
0.1260 0.3602
0.1339 0.3696
0.1417 0.3790
0.1496 0.3884
0.1575 0.3978
0.1654 0.4072
0.1732 0.4167
0.1811 0.4262
0.1890 0.4356
0.1969 0.4451
0.2047 0.4545
0.2126 0.4640
0.2205 0.4735
0.2283 0.4830
0.2362 0.4924
0.2441 0.5019
0.2520 0.5114
0.2598 0.5208
0.2677 0.5303
0.2756 0.5397
0.2835 0.5491
0.2913 0.5585
0.2992 0.5679
0.3071 0.5773
0.3150 0.5867
0.3228 0.5960
0.3307 0.6053
0.3386 0.6146
0.3465 0.6239
0.3543 0.6332
0.3622 0.6424
0.3701 0.6516
0.3780 0.6607
0.3858 0.6699
0.3937 0.6790
0.4016 0.6880
0.4094 0.6970
0.4173 0.7060
0.4252 0.7149
0.4331 0.7238
0.4409 0.7326
0.4488 0.7414
0.4567 0.7501
0.4646 0.7587
0.4724 0.7673
0.4803 0.7758
0.4882 0.7842
0.4961 0.7926
0.5039 0.8009
0.5118 0.8091
0.5197 0.8172
0.5276 0.8252
0.5354 0.8331
0.5433 0.8409
0.5512 0.8485
0.5591 0.8561
0.5669 0.8635
0.5748 0.8707
0.5827 0.8778
0.5906 0.8847
0.5984 0.8915
0.6063 0.8980
0.6142 0.9043
0.6220 0.9104
0.6299 0.9162
0.6378 0.9217
0.6457 0.9269
0.6535 0.9317
0.6614 0.9360
0.6693 0.9399
0.6772 0.9432
0.6850 0.9457
0.6929 0.9473
0.7008 0.9477
0.7087 0.9463
0.7165 0.9417
0.7244 0.9301
0.7323 0.9014
0.7402 0.8289
0.7480 0.6624
0.7559 0.3894
0.7638 0.1538
0.7717 0.0437
0.7795 0.0131
0.7874 0.0107
0.7953 0.0143
0.8031 0.0189
0.8110 0.0240
0.8189 0.0294
0.8268 0.0352
0.8346 0.0412
0.8425 0.0475
0.8504 0.0540
0.8583 0.0608
0.8661 0.0677
0.8740 0.0748
0.8819 0.0821
0.8898 0.0895
0.8976 0.0971
0.9055 0.1048
0.9134 0.1127
0.9213 0.1206
0.9291 0.1287
0.9370 0.1368
0.9449 0.1450
0.9528 0.1534
0.9606 0.1618
0.9685 0.1703
0.9764 0.1788
0.9843 0.1875
0.9921 0.1962
1.0000 0.2049
};
\addplot[blue, thick]
table{
0.0000 0.2071
0.0079 0.2159
0.0157 0.2247
0.0236 0.2336
0.0315 0.2425
0.0394 0.2515
0.0472 0.2605
0.0551 0.2696
0.0630 0.2787
0.0709 0.2879
0.0787 0.2971
0.0866 0.3063
0.0945 0.3156
0.1024 0.3248
0.1102 0.3342
0.1181 0.3435
0.1260 0.3529
0.1339 0.3623
0.1417 0.3717
0.1496 0.3811
0.1575 0.3905
0.1654 0.4000
0.1732 0.4095
0.1811 0.4190
0.1890 0.4285
0.1969 0.4380
0.2047 0.4475
0.2126 0.4570
0.2205 0.4666
0.2283 0.4761
0.2362 0.4857
0.2441 0.4952
0.2520 0.5048
0.2598 0.5143
0.2677 0.5239
0.2756 0.5334
0.2835 0.5430
0.2913 0.5525
0.2992 0.5620
0.3071 0.5716
0.3150 0.5811
0.3228 0.5906
0.3307 0.6000
0.3386 0.6095
0.3465 0.6189
0.3543 0.6284
0.3622 0.6378
0.3701 0.6472
0.3780 0.6566
0.3858 0.6659
0.3937 0.6752
0.4016 0.6844
0.4094 0.6937
0.4173 0.7030
0.4252 0.7122
0.4331 0.7214
0.4409 0.7305
0.4488 0.7395
0.4567 0.7485
0.4646 0.7574
0.4724 0.7664
0.4803 0.7754
0.4882 0.7843
0.4961 0.7931
0.5039 0.8018
0.5118 0.8103
0.5197 0.8188
0.5276 0.8273
0.5354 0.8358
0.5433 0.8443
0.5512 0.8528
0.5591 0.8610
0.5669 0.8690
0.5748 0.8768
0.5827 0.8845
0.5906 0.8921
0.5984 0.8998
0.6063 0.9074
0.6142 0.9151
0.6220 0.9227
0.6299 0.9298
0.6378 0.9365
0.6457 0.9428
0.6535 0.9489
0.6614 0.9550
0.6693 0.9610
0.6772 0.9671
0.6850 0.9733
0.6929 0.9798
0.7008 0.9856
0.7087 0.9889
0.7165 0.9899
0.7244 0.9896
0.7323 0.9843
0.7402 0.9638
0.7480 0.8089
0.7559 0.2330
0.7638 0.0026
0.7717 0.0032
0.7795 0.0059
0.7874 0.0093
0.7953 0.0131
0.8031 0.0175
0.8110 0.0222
0.8189 0.0273
0.8268 0.0328
0.8346 0.0385
0.8425 0.0445
0.8504 0.0508
0.8583 0.0572
0.8661 0.0639
0.8740 0.0708
0.8819 0.0779
0.8898 0.0851
0.8976 0.0925
0.9055 0.1000
0.9134 0.1076
0.9213 0.1154
0.9291 0.1233
0.9370 0.1312
0.9449 0.1393
0.9528 0.1475
0.9606 0.1558
0.9685 0.1642
0.9764 0.1726
0.9843 0.1811
0.9921 0.1897
1.0000 0.1984
};
\end{axis}
\end{tikzpicture}
\caption{Experiment 5. Second-order numerical approximations for the nonlocal model with $\delta=3\Dx$ with $\Dx=1/32$, $\Dx=1/64$, and $\Dx=1/128$ (increasingly darker red, $p=0$) and for the local conservation law with $\Dx=1/128$ (blue) and $\lambda=0.8$, $T=0.5$, and $u(x,0)=u_0^1(x)$.}
\label{fig: Experiment 5}
\end{figure}
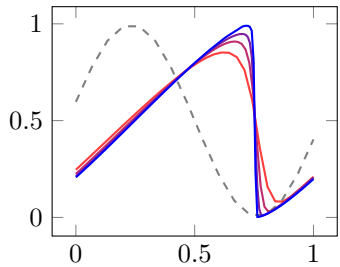

\begin{table}[t]
\centering
\begin{tabular}{@{}rlcc@{}}
  \toprule
  \multicolumn{1}{c}{$n$} & \multicolumn{1}{c}{$\delta$} & $\Lone$ error & $\Lone$ OOC\\
  \midrule
	$  8$ & $0.375$  & $\num{1.689e-01}$ & -- \\
	$ 16$ & $0.1875$ & $\num{1.052e-01}$ & $0.68$ \\
	$ 32$ & $0.09375$ & $\num{6.065e-02}$ & $0.80$ \\
	$ 64$ & $0.046875$ & $\num{3.239e-02}$ & $0.90$ \\
	$128$ & $0.0234375$ & $\num{1.616e-02}$ & $1.00$ \\
	$256$ & $0.01171875$ & $\num{7.747e-03}$ & $1.06$ \\
	$512$ & $0.005859375$ & $\num{3.612e-03}$ & $1.10$ \\
  \bottomrule
\end{tabular}
\caption{Experiment 5. $\Lone$ errors and observed order of convergence of the second-order method for the nonlocal model where $\delta=3\Dx$ with $\lambda=0.8$, $T=0.5$, $u(x,0)=u_0^1(x)$, and $p=0$.}
\label{tbl: Errors and convergence rates experiment 5}
\end{table}

\section{Conclusion}
We have developed and analyzed a second-order accurate numerical method for the nonlocal pair-interaction model. Our numerical method generalizes second-order reconstruction-based schemes for local conservation laws in the sense that, as the nonlocal horizon parameter vanishes, we recover a well-known second-order scheme for the local equation. In contrast to the case of local conservation laws, the second-order scheme we developed converges towards the unique entropy solution provided that the nonlocal interaction kernel satisfies a certain growth condition near zero.

We further proved that weak solutions of the nonlocal pair-interaction model have more regularity as compared to solutions of local conservation laws -- a fact that increases the impact and effectiveness of second-order schemes for the nonlocal model. In particular, we showed that weak solutions of the nonlocal model can only exhibit stationary discontinuities and that traveling wave solutions, when not stationary, are smooth.

Lastly, we provided several numerical experiments comparing our second-order scheme to the first-order scheme presented in~\cite{du2017nonlocal,du2017numerical} and to second-order reconstruction-based schemes for local conservation laws. Notably, we observed a second-order convergence rate for our scheme in regimes where the second-order scheme for the corresponding local conservation law deteriorates to first-order. Our numerical experiments further indicated that the formation of (stationary) shocks in the nonlocal model depends on the choice of flux function and that our scheme is asymptotically compatible with the local entropy solution.

\begin{appendices}
\crefalias{section}{appendix}

\section{Temporal Lipschitz continuity}\label{app:lipschitz}
\begin{lemma}\label{prop:lipschitz}
Let $u\in \mathrm{L}^\infty(\R\times(0,T))$ satisfy
\begin{equation}\label{eq:weaklipschitz}
\int_\R\int_0^T u(x,t)\frac{\partial\phi}{\partial t}(x,t)\diff t\diff x \leq C\|\phi\|_{\mathrm{L}^1((\R\times(0,T)))}
\end{equation}
for all $\phi\in \mathcal{C}^1_c(\R\times(0,T))$. Then there is a function $\widetilde{u}:\R\times(0,T)\to\R$, equal to $u$ almost everywhere, such that $t\mapsto \widetilde{u}(x,t)$ is Lipschitz continuous for a.e.~$x\in\R$.
\end{lemma}
\begin{proof}
Let $K\subset\R$ be any compact set and view $u$ as a function $u\in \mathrm{L}^1(K\times(0,T))$. We identify $u$ with an a.e.~equal function $v \in \mathrm{L}^1(K, \mathrm{L}^1(0,T))$. Then \eqref{eq:weaklipschitz} implies
\[
\int_K \psi(x) \int_0^T \theta'(t) v(x,t)\diff t\diff x \leq C\|\psi\|_{\mathrm{L}^1(K)}\|\theta\|_{\mathrm{L}^1(0,T)}
\]
for all $\psi\in \mathrm{L}^1(K)$, $\theta\in \mathcal{C}^1([0,T])$, which again implies that $\int_0^T \theta'(t) v(x,t)\diff t \leq C\|\theta\|_{\mathrm{L}^1(0,T)}$ for a.e.~$x\in K$. Hence, by Morrey's inequality, there is (for a.e.~$x\in K$) a function $\widetilde{u}=\widetilde{u}(x,t)$ with $\widetilde{u}(x,t)=v(x,t)$ for a.e.~$t\in(0,T)$ such that $t \mapsto \widetilde{u}(x,t)$ is Lipschitz continuous with Lipschitz constant at most $C$. Since $K\subset\R$ was arbitrary, the claim holds for a.e.~$x\in\R$. Finally, it is clear that $\widetilde{u}=u$ a.e.~in $\R\times(0,T)$.
\end{proof}

\section{Integration by parts on a bounded domain}\label{app: integration by parts}
\begin{lemma}\label{lem: Integration by parts}
    Assume that $u\in \mathcal{C}^1(U)$ for some open $U\subset\Omega\coloneqq\R\times\R_+$ and let $\phi\in \mathcal{C}^1_c(\Omega)$. Let $D\subset U$ be bounded with Lipschitz boundary and satisfy $\overline{D}\subset U$. Then
    \begin{equation}
    \begin{split}
    &\iint_{D} \biggl(u\frac{\partial\phi}{\partial t} + \int_0^\delta \frac{\tau_h\phi-\phi}{h}g(u,\tau_hu)\omega_\delta(h)\diff h \biggr)\diff x \diff t \\
    &= -\iint_D \phi\biggl(\frac{\partial u}{\partial t} + \int_0^\delta \frac{g(u,\tau_h u) - g(\tau_{-h}u,u)}{h}\omega_\delta(h)\diff h\biggr) \diff x \diff t\\
    &\mathrel{\hphantom{=}} +\int_{\partial D}u\phi n_t\diff S + \int_0^\delta\iint_\Omega\phi g(\tau_{-h}u,u)\frac{\tau_{-h}\ind_D-\ind_D}{h}\omega_\delta(h)\diff x\diff t\diff h
    \end{split}
    \end{equation}
    where $n_t$ is the $t$-component of the outward pointing normal to $\partial D$.
\end{lemma}
\begin{proof}
The integration by parts of the $u\frac{\partial\phi}{\partial t}$ term is standard. Let $\eps\in(0,\delta)$ and write the spatial term as $E_\eps+\tilde{E}_\eps$, where
\begin{align*}
\tilde{E}_\eps &= \iint_{D}\int_0^\eps \frac{\tau_h\phi-\phi}{h}g(u,\tau_hu)\omega_\delta(h)\diff h \diff x \diff t,\\
E_\eps &= \iint_{D}\int_\eps^\delta \frac{\tau_h\phi-\phi}{h}g(u,\tau_hu)\omega_\delta(h)\diff h \diff x \diff t.
\end{align*}
If $|D|$ denotes the (finite) Lebesgue measure of $D$ then $$|\tilde{E}_\eps|\leq |D|\|g(u,\tau_\cdot u)\|_{\mathrm{L}^\infty}\|\partial_x\phi\|_{\mathrm{L}^\infty}\|\omega_\delta\|_{\mathrm{L}^1((0,\eps))}\to0$$ as $\eps\to0$, and
\begin{align*}
E_\eps &= \iint_{\Omega}\ind_D\int_\eps^\delta \frac{\tau_h\phi}{h}g(u,\tau_hu)\omega_\delta(h)\diff h \diff x \diff t - \iint_\Omega\ind_D\int_\eps^\delta \frac{\phi}{h}g(u,\tau_hu)\omega_\delta(h)\diff h \diff x \diff t \\
&= \iint_{\Omega}\phi\int_\eps^\delta\frac{g(\tau_{-h}u,u)\tau_{-h}\ind_D - g(u,\tau_hu)\ind_D}{h}\omega_\delta(h)\diff h \diff x \diff t \\
&= \iint_D\phi\int_\eps^\delta\frac{g(\tau_{-h}u,u) - g(u,\tau_hu)}{h}\omega_\delta(h)\diff h \diff x \diff t \\
&\mathrel{\hphantom{=}}+ \iint_{\Omega}\phi\int_\eps^\delta g(\tau_{-h}u,u)\frac{\tau_{-h}\ind_D-\ind_D}{h}\omega_\delta(h)\diff h \diff x \diff t \\
&\eqqcolon F_{\eps,1}+F_{\eps,2}
\end{align*}
We claim that
\begin{align*}
F_{\eps,1}&\to F_1\coloneqq\iint_D\phi\int_0^\delta\frac{g(\tau_{-h}u,u) - g(u,\tau_hu)}{h}\omega_\delta(h)\diff h \diff x \diff t, \\
F_{\eps,2}&\to F_2\coloneqq \int_0^\delta\iint_\Omega\phi g(\tau_{-h}u,u)\frac{\tau_{-h}\ind_D-\ind_D}{h}\diff x\diff t\diff h
\end{align*}
as $\eps\to0$. Indeed, if $0<\eps<\dist(D, \partial U)$ then
\[
|F_{\eps,1}-F_1| \leq 2\|\phi\|_{\mathrm{L}^1}\bigl(\|\partial_1 g(u,u)\|_{\mathrm{L}^\infty}+\|\partial_2 g(u,u)\|_{\mathrm{L}^\infty}\bigr)\big\|\tfrac{\partial u}{\partial x}\big\|_{\mathrm{L}^\infty(U)}\|\omega_\delta\|_{\mathrm{L}^1((0,\eps))} \to 0
\]
as $\eps\to0$. Similarly (letting $\ominus$ denote the symmetric difference between sets),
\begin{align*}
|F_{\eps,2}-F_2| &\leq \int_0^\eps \iint_{(\tau_{-h} D)\ominus D} \frac{1}{h}|\phi||g(\tau_{-h}u,u)|\omega_\delta(h)\diff x \diff t \diff h \\
&\leq \|\phi\|_{\mathrm{L}^\infty}\|g(\tau_\cdot u,u)\|_{\mathrm{L}^\infty}\int_0^\eps h|\partial D| \frac{1}{h}\omega_\delta(h)\diff h\\
&= \|\phi\|_{\mathrm{L}^\infty}\|g(\tau_\cdot u,u)\|_{\mathrm{L}^\infty}|\partial D| \|\omega_\delta\|_{\mathrm{L}^1((0,\eps))} \to 0
\end{align*}
as $\eps\to0$, where $|\partial D|$ denotes the length of $\partial D$.
\end{proof}

\end{appendices}





\subsection*{Acknowledgments}
We want to thank Siddhartha Mishra and Espen Sande for many insightful discussions.

\bibliographystyle{siamplain}

\end{document}